\documentclass[]{interact}

\usepackage{epstopdf}
\usepackage[caption=false]{subfig}

\usepackage[numbers,sort&compress]{natbib}
\bibpunct[, ]{[}{]}{,}{n}{,}{,}
\renewcommand\bibfont{\fontsize{10}{12}\selectfont}
\makeatletter
\def\NAT@def@citea{\def\@citea{\NAT@separator}}
\makeatother

\theoremstyle{plain}
\newtheorem{theorem}{Theorem}[section]
\newtheorem{lemma}[theorem]{Lemma}

\theoremstyle{definition}
\newtheorem{definition}[theorem]{Definition}
\newtheorem{example}[theorem]{Example}

\theoremstyle{remark}
\newtheorem{remark}[theorem]{Remark}

\newcommand{\eps}{\varepsilon}

\newcommand{\C}{\mathbb{C}}
\newcommand{\R}{\mathbb{R}}

\newcommand{\cS}{\mathcal{S}}
\newcommand{\oeps}{\eps^*}
\newcommand{\bA}{\bold{A}}
\newcommand{\bDelta}{\bold{\Delta}}
\newtheorem{assum}[theorem]{Assumption}
\usepackage{algorithm}
\usepackage{algpseudocode}
\usepackage[hidelinks]{hyperref}
\usepackage{xcolor}
\usepackage{comment}

\newcommand{\bng}{\color{blue}}
\newcommand{\eng}{\color{black}}

\begin{document}

\articletype{Research article}

\title{Approximating the closest structured singular matrix polynomial}

\author{
\name{Miryam Gnazzo \textsuperscript{a}\thanks{CONTACT Miryam Gnazzo. Email: miryam.gnazzo@gssi.it} and Nicola Guglielmi\textsuperscript{a}}
\affil{\textsuperscript{a} Division of Mathematics, Gran Sasso Science Institute, L'Aquila, Italy.}
}

\maketitle

\begin{abstract}
Consider a matrix polynomial $P \left( \lambda \right)= A_0 + \lambda A_1 + \ldots + \lambda^d A_d$, with $A_0,\ldots, A_d$ complex (or real) matrices with a certain structure. In this paper we discuss an iterative method to numerically approximate the closest structured singular matrix polynomial 
$\widetilde P\left( \lambda \right)$, using the distance induced by the Frobenius norm. An important peculiarity of the approach we propose is the possibility to include different types of structural constraints. The method also allows us to limit the perturbations to just a few matrices and also to include additional structures, such as the preservation of the sparsity pattern of one or more matrices $A_i$, and also collective-like properties, like a palindromic structure. The iterative method is based on the numerical integration of the gradient system associated with a suitable functional which quantifies the distance to singularity of a matrix polynomial.
\end{abstract}


\section{Introduction}
Let $A_0,\ldots, A_d \in \mathbb{C}^{n\times n}$ form a matrix polynomial
\begin{equation*}
P \left( \lambda \right)=\sum_{i=0}^d \lambda^i A_i.
\end{equation*}
$P\left( \lambda \right)$ is called regular if $\det\left( P \left( \lambda \right) \right)$ does not vanish identically for all $\lambda \in \C$, otherwise, it is called singular. Consider for example a matrix pencil, $A_0 + \lambda A_1$. The notion of singularity is fundamental in the theory of linear differential-algebraic equations $A_1\dot x(t) + A_0 x(t) + f(t) = 0$ (which might arise as linearizations of nonlinear differential-algebraic equations $F(x,\dot x)=0$ near a stationary point). If the matrix pencil $A_0 + \lambda A_1$ is singular, then there exists no initial value $x(0)$ such that the corresponding initial value problem has a unique solution.

Note that in several applications of interest to dynamical systems on networks, $A_1$ is typically an adjacency matrix or related fixed matrix depending only on the network topology and hence is not subject to perturbations.

In the more general case of a matrix polynomial of degree $d$, an interesting problem is that of understanding if there exists a nearby singular matrix polynomial to $P$ of the same degree. 
The norm of the smallest perturbation of $P$ which makes it singular is called the \emph{distance to singularity} of $P$.
Similarly to the case of matrix pencils, estimating the value of this distance is important in various applications,
in particular, in the analysis of differential-algebraic equations,
\begin{equation}\label{nldae}
F\left(t,x(t),x'(t),\ldots, x^{\left(d \right)}(t) \right)=0,
\end{equation}
where $x^{\left( k \right)}(t)$ denotes the $k$-th time derivative of the state $x$, see e.g.
\cite{HaiLubRoc, BreCamPetz}. If DAEs are linearized
along a stationary solution, \cite{Campbell}, then one obtains a  linear system of DAEs
\begin{equation}\label{dae}
\sum\limits_{i=0}^{d}  A_i x^{(i)}(t) = 0,
\end{equation}
with constant coefficients $A_0,\ldots, A_d \in\mathbb C^{n\times n}$. If the associated matrix polynomial is singular, then the initial value problem of solving (\ref{dae}) with a consistent initial value $x^{\left( k \right)}(0)=x_0^k$, ($k=0,1,\ldots,d-1)$ is not solvable and/or the solution is not unique.
Thus, a singular matrix polynomial indicates that (\ref{nldae}) is not well-posed, and the distance to singularity of $P$ is a robustness measure for the DAE in the neighbourhood of a stationary solution. 

\subsection*{The considered problem}
Consider a regular matrix polynomial $P\left( \lambda \right)= A_0 + \lambda A_1 + \ldots + \lambda^d A_d$. We are interested to know how far it is from a singular matrix polynomial.
In this paper we consider structured perturbations 
\[
\mathbf{\Delta A}:=
\begin{bmatrix}
    \Delta A_0 \\
    \vdots\\
    \Delta A_d 
\end{bmatrix} \in \mathcal{S} \subseteq \mathbb{C}^{\left( d+1 \right)n \times n} \quad \mbox{to the coefficients} \;
\begin{bmatrix}
    A_0 \\
    \vdots\\
    A_d 
\end{bmatrix}, 
\] 
where $\Delta A_i \in \mathbb{C}^{n \times n}$, for $i=0,\ldots,d$ and the structure space $\cS$ is an arbitrary complex- or real-linear subspace of $\C^{\left( d+1\right) n \times n}$. For example, $\cS$ might be a space of real or complex matrices with a given sparsity pattern. We mainly consider the following matrix nearness problems (and a few variants):

{\bf Problem.} \ {\it Find $\mathbf{\Delta A} \in \cS$ of minimal Frobenius norm such that 
$(A_0+\Delta A_0) + \lambda(A_1+\Delta A_1) + \ldots + \lambda^d (A_d + \Delta A_d)$ is singular.}

By minimal Frobenius norm we intend that we aim to minimize 
\[
\| [ \Delta A_0, \Delta A_1, \ldots, \Delta A_d ] \|_F,
\]
which determines the \emph{distance to singularity} of $P\left( \lambda \right)$, with respect to the Frobenius norm. We will provide a more rigorous definition in Section \ref{sec: section ODE based approach}. 

In this general setting, we may recognize a few situations in which the singularity is determined by properties of the coefficients of the polynomial. For example, in the case of matrix pencils, a necessary condition for $A_0 + \lambda A_1$ to be singular is clearly that both $A_0$ and $A_1$ are singular matrices. A sufficient condition for $A_0 + \lambda A_1$ to be singular is that $A_0$ and $A_1$ have a common nonzero vector in their kernels. While this is a special case of particular interest, the Kronecker normal form of a matrix pencil (see Gantmacher~\cite{Gantmacher}) shows that a common null-vector is not a necessary condition for a matrix pencil to be singular. 

In this work, we focus on two possible settings for the singularity of matrix polynomials. The first possibility occurs when the coefficient matrices of the matrix polynomial have a common nontrivial kernel, that is,
\begin{equation}
\label{eq:vector_comm_ker}
    \exists x \neq 0 \; \mbox{such that} \; A_ix=0, \; \mbox{for } i=0,\ldots,d.
\end{equation}
The structured analysis presented in \cite{MehlMehrWoj} proves that this possibility takes place in the case of dissipative Hamiltonian matrix pencils and their corresponding generalization to matrix polynomials.
The second possibility occurs when the condition \eqref{eq:vector_comm_ker} is not satisfied.  In this case, an example is given by the quadratic matrix polynomial:
\begin{equation}
\label{pol_P1_introduction}
 P_{1} \left( \lambda \right)= \lambda^2 \left[ \begin{array}{c c}
    1  & 0 \\
    0  & 0
 \end{array} \right] + \lambda \left[ \begin{array}{c c}
    0  & 1 \\
    1  & 0
 \end{array} \right] + \left[ \begin{array}{c c}
    0  & 0 \\
    0  & 1
 \end{array} \right].
\end{equation}
It is straightforward to show that the determinant of $P_{1}\left( \lambda \right)$ is identically zero, but the coefficients do not have a common nontrivial kernel (in particular one of them is non-singular). The presence of a common kernel in the case of unstructured perturbations can be easily solved. Instead in the nontrivial case of structured perturbations the presence of a common kernel may suggest a different approach in the computation of the distance to singularity. In particular, in this setting, it is more appropriate to refer to the distance to the nearest polynomial whose coefficients have a nontrivial common kernel.

In situations where we do not impose additional structures, it is possible to compute the unstructured distance to singularity of a regular matrix polynomial. If we restrict the setting and consider only pencils, we may find a wide literature, see for instance \cite{ByHeMeh,DopicoNoferiniNyman}. Instead for the case of polynomials with a degree higher than $1$, first a method based on structured perturbations of block Toeplitz matrices containing the coefficients of the matrix polynomial has been introduced in \cite{GiesHaral}. Then recently Das and Bora \cite{DasBora} proposed a characterization of the nearest singular matrix polynomial to a regular one, using a method that involves convolution matrices. While their numerical method may be employed to compute the approximate distance to singularity, their theoretical results lead to a few upper bounds for the unstructured distance to singularity. We may also apply their bounds for testing the quality of our results, as we explain in Section \ref{sec:a posteriori}.
In this paper, we focus on the computation of a numerical upper bound for the structured distance to singularity. For this purpose, given a regular matrix polynomial, whose coefficients have a certain structure, we provide a method that computes a nearby matrix polynomial with the same degree and structure. This work aims to expand the recent article on matrix pencils \cite{GugLubMeh}, on one side extending the analysis from pencils to general polynomials and on the other side concentrating on both individual and collective structures of the coefficient matrices. To the best of our knowledge, this is the first attempt in the literature to propose a method for the computation of the structured distance to singularity, suitable for several situations and different kinds of additional structures on the coefficients.

The article is organized as follows. In Section \ref{sec: section ODE based approach} we provide several starting notions and formulate the considered matrix nearness problem as an optimization problem. In Section \ref{sec:Structured_distance to singularity} we propose a method for the numerical approximation of the structured distance to singularity. We introduce a functional associated with the optimization problem and propose a two-level iterative method for its numerical minimization. Then in Section \ref{sec:computationa-approach} we make some computational considerations and show some illustrative numerical examples. Finally in Section \ref{sec:distance_for_common_kernel} we address the problem of finding the closest polynomial whose coefficient matrices share a left/right common kernel, where we propose a method ad hoc; a few numerical experiments illustrate the behaviour of the method and conclude the article.

\section{An associated optimization problem}
\label{sec: section ODE based approach}

We use the following notation: given two matrices $A, B \in \mathbb{C}^{n  \times n}$, we denote by 
\[
\left\langle A,B \right\rangle = \mbox{trace} \left( A^H B \right)
\]
the Frobenius inner product on $\mathbb{C}^{n \times n}$ and by $\| A \|_F=\left(\sum_{i,j=1}^n \left| a_{ij}\right|^2 \right)^\frac{1}{2}$ the associated norm. Here $A^H$ indicates the conjugate transpose of the matrix $A$. We denote by $\mathcal{P}_d$ the set of square matrix polynomials of size $n \times n$ and degree $d$; i.e.
\begin{equation*} 
\label{eq:pol}
    P(\lambda)=\sum_{i=0}^d \lambda^i A_i, \quad A_i \in \mathbb{C}^{n \times n},
\end{equation*}
where $A_d$ is not identically zero. We associate to the polynomial $P(\lambda)$ the rectangular matrix 
\begin{equation}
\label{def:Polynomial_P}
\bold{A}= \left[ \begin{array}{c}
     A_d  \\
     A_{d-1}\\
     \vdots \\
     A_0
\end{array} \right]
\end{equation}
and we formally define $\| P \|= \| \bold{A} \|_F$.

\begin{definition}
\label{definition:distance_to_sing_unstr}
Given a matrix polynomial of degree $d$, $P \left( \lambda \right) \in \mathcal{P}_d$, we define the \emph{distance to singularity} as
\begin{equation*}
    d_{\mbox{sing}} \left( P  \right) = \min \left\lbrace \| \Delta P  \| : \; P (\lambda)  + \Delta P (\lambda)   \in \mathcal{P}_d \; \mbox{singular}  \right\rbrace.
\end{equation*}
Here $\Delta P \left( \lambda \right) = \sum_{i=0}^d \lambda^i \Delta A_i \ \in \mathcal{P}_d$.
\end{definition}

A matrix polynomial is singular if and only if $\det(P(\lambda))=0$ for all $\lambda$. Thus we look for $\Delta P \left( \lambda \right)$ such that
\begin{equation}
     \det \left( P \left( \lambda \right) + \Delta P \left( \lambda \right) \right) \equiv 0.
    \label{eq:determinant}
\end{equation}
Since the determinant  \eqref{eq:determinant}  is a scalar polynomial of degree lower or equal than $dn$, using the characterization given by the fundamental theorem of algebra, we have that the previous condition is equivalent for $m \geq dn+1$ to 
\begin{equation*}
    \sigma_{\min} \left( \mu_j^d \left( A_d + \Delta A_d \right) + 
    \ldots + \left( A_0 + \Delta A_0 \right) \right)=0, \quad j=1,\ldots,m,
\end{equation*}
with $\mu_1, \ldots , \mu_m$ distinct complex points and $\sigma_{\min}$ denotes the smallest singular value. Therefore we consider the optimization problem
\begin{align*}
\Delta P ^* &= \mbox{arg} \min_{\Delta P \in \mathcal{P}_d} \|\Delta P \|  \notag \\
& \mbox{subj. to} \; \sigma_{\min} \Bigl(  (P+\Delta P) (\mu_j) \Bigr) =0, \quad j=1,\ldots,m,
\label{eq: opt_pbm}
\end{align*}
where $\Delta P^*$ is associated with the array of perturbations
\begin{equation*}
    \Delta \bold{A}^*= \left[ \begin{array}{c}
     \Delta A_d^* \\
     \Delta A_{d-1}^*\\
     \vdots \\
     \Delta A_0^*
\end{array} \right].
\end{equation*}
With this formulation, we have that $d_{\mbox{sing}} \left( P  \right) = \| \Delta \bold{A}^* \|_F$.

\begin{remark}
The choice of the set of points $\mu_1, \ldots, \mu_m$ could affect the numerical resolution of the optimization problem. One possible choice is the set 
\[
\mu_j= \rho \, e^{\frac{2j\pi}{m}\textbf{i}}, \quad j=1,\ldots,m, \; \mbox{and} \; \rho >0
\]
where $\rho$ scales with the magnitude of the entries of $P\left(\lambda \right)$. Nevertheless, in several situations there is the possibility of more natural choices for $\mu_j$. For instance, in the presence of matrix polynomials with real entries, a reasonable choice may be the use of Chebyshev points. A deeper investigation of the influence of the set of points on the results is missing at the moment, but it is out of the scope of this article.
\end{remark}

\begin{remark}
The choice of employing the smallest singular values of the considered matrices is motivated by the aim to compute a numerically singular matrix (see for instance Chapter $5$ of \cite{GolubLoan}). Since our method relies on the numerical approximation of the solution of a non-convex optimization problem, we underline that the computed quantity is a numerically approximated upper bound of the distance to singularity, in the sense that it provides a nearby matrix polynomial which is numerically singular, whose distance from the given polynomial is larger or equal to the minimal one.

Moreover, in this paper we make use of the singular values of the matrices, instead of the eigenvalues, as proposed in \cite{GugLubMeh}. This avoids computational issues associated with the possible bad conditioning of the eigenvalues. 

As for singular vectors, we can monitor their conditioning by controlling the distance of singular values. Indeed a formula for the expansion of the right singular vector $v_1(\varepsilon)$ of the matrix $A(\varepsilon)$, associated with $\sigma_1(\varepsilon)$ is
\[
v_1(\varepsilon)= v_1 + \varepsilon \sum_{i=2}^n \frac{v_i^H A^H A v_i}{(\sigma_1^2-\sigma_i^2)v_i^Hv_i} (v_i -v_1v_1^Hv_i ) + \mathcal{O}(\varepsilon^2),
\]
where $\sigma_i$ and $v_i$ for $i=1,\ldots,n$ are the singular values and the associated right singular vectors, respectively, of the unperturbed matrix $A$. From a numerical point of view - however - what mainly matters is the conditioning at stationary points (and not along the trajectory) since (i) coalescence would not be detected numerically due to non-genericity and (ii) there is no need to follow the exact solution of the gradient system we associate the optimization problem. In case of ill-conditioning at stationary points, special care should be
directed to the eigenvalue solver, possibly augmenting the computational precision.
\end{remark}

\section{Structured distance to singularity}
\label{sec:Structured_distance to singularity}

In several contexts, it can be very useful to perform a structured analysis of the perturbations, allowing only perturbation matrices that respect the possible additional structures of the problem. In this Section, we propose a method for the numerical approximation of the structured distance to singularity. The strategy we propose may also be applied in the unstructured case. In this context, we compute the structured distance to singularity, defined in the following way:

\begin{definition}
\label{def:distance_to_sing_Structured}
Given a subset $\mathcal{S} \subseteq \mathbb{C}^{\left( d+1\right) n \times n}$. The structured distance to the nearest singular polynomial is given by
\begin{equation*}
    d_{\rm sing}^{\mathcal{S}} \left( P \right):=\min \left\lbrace \| \Delta P \|: \; P \left( \lambda \right) + \Delta P \left( \lambda \right) \in \mathcal{P}_d \; \mbox{singular}, \; \Delta \bold{A} \in \mathcal{S} \right\rbrace.
\end{equation*}
\end{definition}

The subset $\mathcal{S} \subseteq \mathbb{C}^{\left( d+1\right)n \times n}$ denotes the set of coefficient matrices with a prescribed additional structure. This structure may involve the single coefficients, including sparsity patterns, real entries, or the whole polynomial. We will give a few examples of possible additional structures in the following subsections. It is straightforward to observe that the unstructured distance to singularity may be seen as a particular situation in which the perturbations belong to $\mathcal{S}=\mathbb{C}^{\left(d+1\right)n\times n}$. We consider this case in Subsection \ref{subsec:Linear_structure}.

We employ the orthogonal projection $\Pi_{\mathcal{S}}$ onto the subspace $\mathcal{S}$ in order to allow only perturbations belonging to $\mathcal{S}$. Note that for $\Pi_{\mathcal{S}}$ orthogonal projection with respect to the Frobenius inner product onto $\mathcal{S}$, we have that: for every $\bold{Z} \in \mathbb{C}^{\left( d+1\right)n \times n}$
\begin{equation}
\label{eq:prop_orthogonal_projection}
    \Pi_{\mathcal{S}}\left(\bold{Z} \right) \in \mathcal{S}, \quad {\rm Re} \left\langle \Pi_{\mathcal{S}} \left( \bold{Z} \right), \bold{W}  \right\rangle = {\rm Re} \left\langle \bold{Z}, \bold{W}  \right\rangle, \; \forall \bold{W} \in \mathcal{S}.
\end{equation}

In a few situations, the projection may be straightforward to construct. For instance, consider $\mathcal{S}$ the space of matrices with real entries, then the projection $\Pi_{\mathcal{S}}\left(\bold{Z} \right)={\rm Re} \left( \bold{Z} \right)$. Instead, for $\mathcal{S}$ matrices with a sparsity pattern, the projection consists of the identity for the entries on the sparsity pattern and the null function on the remaining ones. Nevertheless, in the presence of different matrix spaces $\mathcal{S}$, the computation can be quite delicate (see Subsection \ref{subsec:Structure_on_whole_pol}, for instance).

In an analogous way to the approach in Section \ref{sec: section ODE based approach}, we construct the following optimization problem:
\begin{align}
\label{eq:opt_pbm_Struct_general}
\Delta \bold{A} ^* &= \mbox{arg} \min_{\Delta \bold{A} \in \mathcal{S}} \|\Delta \bold{A} \|_F  \\
& \mbox{subj. to} \; \sigma_{\min} \Bigl(  (P+\Delta P) (\mu_j) \Bigr) =0, \quad j=1,\ldots,m, \notag
\end{align}
with $\mu_1,\ldots,\mu_m$ distinct complex points.

\subsection{Methodology}

In order to solve the optimization problem \eqref{eq:opt_pbm_Struct_general}, the strategy we propose consists of rephrasing the problem into an equivalent one and solving it numerically by a two-step procedure. First we fix the value $\varepsilon >0$ of the norm of the perturbation and rewrite the perturbation  $\Delta \bold{A} = \eps \bDelta$ such that $\| \bDelta \|_F=1$ and $\bDelta\in \mathcal{S}$, with
\begin{equation*}
    \bDelta= \left[ \begin{array}{c}
     \Delta_d \\
     \Delta_{d-1}\\
     \vdots \\
     \Delta_0
\end{array} \right].
\end{equation*}

Since our goal is to find the smallest $\varepsilon$ that solves the equation
\begin{equation*}
    \sigma_{\min} \left( \mu_j^d \left( A_d + \varepsilon \Delta_d \right) +  \mu_j^{d-1} \left( A_{d-1} + \varepsilon \Delta_{d-1} \right) + \ldots + \left( A_0 + \varepsilon \Delta_0 \right) \right)=0,
\end{equation*}
for each $\mu_j, \; j=1, \ldots, m$, then a natural optimization problem associated to \eqref{eq:opt_pbm_Struct_general} is the following.

After defining the functional
\begin{equation}
    G_{\varepsilon} \left( \bDelta \right) = \frac{1}{2} \sum_{j=1}^m \sigma_{j}^2 \left(  \varepsilon, \bDelta \right),
    \label{functional_G}
\end{equation}
where $\sigma_{j} \left( \varepsilon , \bDelta \right)$ is the smallest singular value of the matrix 
\[
\mu_j^d \left( A_d+ \varepsilon \Delta_d \right) +\ldots +\left( A_0 + \varepsilon \Delta_0 \right),
\]
we aim to compute
\[
\oeps = \min\{\eps > 0 \,:\, G_{\varepsilon} \left( \bDelta \right)=0 \}.
\]

Our approach is summarized by the following two-level method:
\begin{itemize}

\item {\bf Inner iteration:\/} Given $\eps>0$, we aim to compute a  matrix $\bDelta(\eps) \in\cS$  
of unit Frobenius norm,  such that 
$G_\eps(\bDelta)$ is minimized, i.e. 
\begin{equation} \label{D-eps}
\bDelta(\eps) = \arg\min\limits_{\bDelta \in \cS, \| \bDelta \|_F = 1} G_\eps(\bDelta).
\end{equation}


\item {\bf Outer iteration:\/} We compute the smallest positive value $\oeps$ with
\begin{equation*} \label{eq:zero} 
g(\oeps)= 0,
\end{equation*}
where $g(\eps)=  G_{\varepsilon} \left( \bDelta(\eps) \right)$.
\end{itemize}

The two-level approach taken here uses an {\it inner iteration} to compute the solution of the optimization 
problem \eqref{D-eps} for a fixed perturbation size $\eps$, and then determines the optimal perturbation size $\oeps$ in an {\it outer iteration}.

The algorithm is not guaranteed to find the global optimum of these nonsmooth and nonconvex optimization problems, but it computes a matrix with the desired spectral property which is locally nearest and often, as observed in our numerical experiments, has a distance close to the minimal distance. In any case it provides an upper bound to the minimal distance, and usually a very tight one.
 Running the algorithm with several different
starting values reduces the risk of getting stuck in a local optimum.

\subsection{The inner iteration}
\label{subsection: inner iteration_struc}

 For simplicity of the notations, we indicate the singular values as $\sigma_j$, without the explicit dependence from $\varepsilon$. In order to minimize the functional \eqref{functional_G}, we use a constrained steepest descent method. We construct the perturbation matrices in the form $\varepsilon \Delta_i \left( t \right)$, depending on the parameter $t$. To satisfy the requirement $\| \bDelta (t) \|_F=1$ for all $t$, we impose
that
\begin{equation*}
    0= \frac{d}{dt} \| \bDelta \|_F^2= \frac{d}{dt} \left\langle \bDelta, \bDelta \right\rangle = \left\langle \dot{\bDelta}, \bDelta \right\rangle + \left\langle \bDelta, \dot{\bDelta} \right\rangle = 2 \mbox{Re} \left\langle \bDelta, \dot{\bDelta} \right\rangle,
\end{equation*}
from which we deduce the condition
\begin{equation}
\label{eq:Condition_conservation_norm}
    \mbox{Re} \left\langle \bDelta, \dot{\bDelta}\right\rangle  =0,
\end{equation}
where we have omitted the dependence of $t$ for conciseness. In order to compute the derivative of the functional \eqref{functional_G}, with respect to $t$, we use the standard result for eigenvalues (see e.g. \cite{Horn}, page $453$ for its extension to singular values), 
from which we can derive the following lemma for singular values:

\begin{lemma}
\label{lemma:derivative_sing_val}
Consider a continuously differentiable path of matrices $D(t)\in\C^{p \times q}$ for $t$ in an open interval $I$. Let $\sigma(t)$, $t\in I$, be a path of simple singular values of $D(t)$, with $\sigma(t) \neq 0$, for all $t \in I$.  Let $u(t)$ and $v(t)$ be left and right singular vectors of $D(t)$ to the singular value $\sigma(t)$, that is, $D(t) v(t) = \sigma(t) u(t)$ and $u(t)^H D(t)= \sigma(t) v(t)^H$ with $\|u(t)\|=\|v(t)\|=1$.
Then, $\sigma$ is differentiable on $I$ with the derivative
\begin{equation*}
\dot\sigma(t) = {\rm{Re}} \left(u(t)^H \dot D(t) v(t) \right).
\end{equation*}
\end{lemma}

Choosing 
\begin{equation*}
D\left( t \right)=\mu_j^d  \left( A_d + \varepsilon \Delta_d (t) \right) + \mu_j^{d-1} \left( A_{d-1} + \varepsilon \Delta_{d-1}(t) \right) + \ldots + \left( A_0 + \varepsilon \Delta_0(t) \right)
\end{equation*}
in Lemma \ref{lemma:derivative_sing_val} for each $j=1,\ldots, m$, we have that for $\sigma_j \left( t \right)$ simple singular value, with $u_j$ and $v_j$ left and right singular vectors, respectively:
\begin{align*}
    \frac{1}{2}\frac{d}{dt} \sigma_j^2 = \sigma_j \dot{ \sigma_j} &= \sigma_j \varepsilon \mbox{Re} \left( u_j^H \left(\sum_{k=0}^d \mu_j^k \dot{\Delta}_k \right) v_j \right)  \notag \\
    &= \varepsilon \left( \mbox{Re} \left\langle\sigma_j \Bar{\mu_j}^d u_j v_j^H, \dot{\Delta}_d \right\rangle +\ldots + \mbox{Re} \left\langle\sigma_j u_j v_j^H, \dot{\Delta}_0 \right\rangle \right).
\end{align*}
Note that the assumption that $\sigma_j$ is simple and nonzero for every $j$ is a generic property. Therefore we obtain that the derivative of \eqref{functional_G} can be expressed using the Frobenius inner product as follows:
\begin{equation}
    \frac{d}{dt} G_{\varepsilon} \left( t\right)= \varepsilon \left( \mbox{Re} \left\langle M_d, \dot{\Delta}_d \right\rangle + \mbox{Re} \left\langle M_{d-1}, \dot{\Delta}_{d-1} \right\rangle +\ldots + \mbox{Re} \left\langle M_0, \dot{\Delta}_0 \right\rangle \right),
    \label{eq:derivative_functional}
\end{equation}
where $M_k:=\sum_{j=1}^m \bar{\mu_j}^k \sigma_j u_j v_j^H$, for each $k=0,\ldots,d$. Relation \eqref{eq:derivative_functional} yields
\begin{equation*}
\frac{1}{\varepsilon} \frac{d}{dt} G_{\varepsilon} \left( t \right)={\rm Re} \left\langle \bold{M}, \dot{\bDelta} \right\rangle.
\end{equation*}
where 
\begin{equation*}
    \bold{M}=\left[ \begin{array}{c}
         M_d  \\
         M_{d-1} \\
         \vdots \\
         M_0
    \end{array} \right].
\end{equation*}

The structure requirements $\bDelta \in \mathcal{S}$ enter in the minimization of the functional \eqref{functional_G} as an additional constraint in the minimization problem. We get that given a path of matrices $\bDelta \left( t \right) \in \mathcal{S}$, then also $\dot{\bDelta}$ belongs to $\mathcal{S}$ and, starting from the derivative \eqref{eq:derivative_functional}, we have that
\begin{equation}
\label{eq:derivative_functional_G_projected}
    \frac{1}{\varepsilon} \frac{d}{d t} G_{\varepsilon} \left( \bDelta \right)= {\rm Re} \left\langle \Pi_{\mathcal{S}} \left( \bold{M} \right), \dot{\bDelta} \right\rangle,
\end{equation}
where we use the property \eqref{eq:prop_orthogonal_projection} of the orthogonal projection $\Pi_{\mathcal{S}}$. From these computations, we get the following result:

\begin{lemma}
\label{lem: minimization_struct_pert_general}
Consider $\bold{M} \in \mathbb{C}^{\left(d+1 \right) n \times  n}$ and $\bDelta \in \mathcal{S}$ s.t. $\| \bDelta \|_F=1$, not proportional to each other, and let $\bold{Z} \in \mathcal{S}$. A solution of the minimization problem
\begin{align}
\label{eq:solution_opt_pbm_struct_general}
    \bold{Z}^* &= \arg \min_{\bold{Z} \in \mathcal{S}} {\rm Re} \left\langle \Pi_{\mathcal{S}} \left( \bold{M} \right) , \bold{Z} \right\rangle \\
    & {\rm{subj. \; to}} \; {\rm{Re}} \left\langle \bDelta, \bold{Z} \right\rangle =0 \; {\rm{and}} \; \|  \bold{Z} \|_F=1, \notag
\end{align}
is given by 
\begin{equation} \label{eq:Zopt}
    \kappa \bold{Z}^* = -   \Pi_{\mathcal{S}} \left( \bold{M} \right) + \eta \bDelta,
\end{equation}
where $\eta={\rm Re} \left\langle \bDelta, \Pi_{\mathcal{S}} \left( \bold{M} \right) \right\rangle$ and $\kappa$ is the norm of the right-hand side.
\end{lemma}
\begin{proof}
The result follows by noting that the real part of the complex inner product on $\C^{p \times q}$ is a real inner product on $\R^{2p \times 2q}$, where in our case $p:=\left( d+1\right)n$ and $q:=m$, and the real inner product with a given vector (which here is a matrix) is maximized over a subspace by orthogonally projecting the vector onto that subspace. The expression in (\ref{eq:Zopt}) is the orthogonal projection of $-\Pi_{\cS} (\bold{M})$ to the orthogonal complement of the span of $\bDelta$, which is the tangent space at $\bDelta$ of the manifold of matrices of unit Frobenius norm.
Since $\bDelta,\Pi_{\cS} (\bold{M}) \in\cS$ in \eqref{eq:Zopt}, also $\mathbf{Z}^*$ is in $\cS$.
The norm conservation follows from $-{\rm Re} \langle   \bDelta, \Pi_{\mathcal{S}} \left( \bold{M} \right) \rangle + \eta \langle \bDelta, \bDelta \rangle = -{\rm Re} \langle   \bDelta, \Pi_{\mathcal{S}} \left( \bold{M} \right) \rangle + \eta = 0$. 
\end{proof}
For the functional $G_{\varepsilon}(\mathbf{\Delta})$, we have that the optimal steepest descent $\mathbf{Z}=\dot{\mathbf{\Delta}}$ is given by the solution of the constrained minimization problem \eqref{eq:solution_opt_pbm_struct_general}. Motivated by Lemma \ref{lem: minimization_struct_pert_general}, we consider the constrained gradient system
\begin{equation}
\label{eq:gradient_system_structured}
\dot{\bDelta}= - \Pi_{\mathcal{S}} \left( \bold{M} \right) + \eta \bDelta
\end{equation}
on the manifold of structured matrices in $\mathcal{S}$, of unit Frobenius norm. 
We prove now a sequence of results that allow us to optimize
the functional and characterize the stationary points of system \eqref{eq:gradient_system_structured} as (local) extremizers.

The first result (Theorem \ref{thm:G_decreases}) states that along solution trajectories of \eqref{eq:gradient_system_structured} the functional \eqref{functional_G} decreases monotonically in $t$.

The second one (Theorem \ref{thm:equivalence_charact}) characterizes stationary points
of the system \eqref{eq:gradient_system_structured} as real multiples of the gradient $\bold{M}$.

The third one (Lemma \ref{lem:nonzero-projected-gradient}) assures the non-vanishing property of the gradient which guarantees the applicability of Theorem \ref{thm:equivalence_charact} and of the integration of system \eqref{eq:gradient_system_structured} to compute solutions to the considered 
optimization problem.

\begin{theorem}
\label{thm:G_decreases}
Consider $\bDelta \left( t \right) \in \mathcal{S} \subseteq \mathbb{C}^{\left( d+1 \right) n \times n}$, with $\| \bDelta \left( t \right) \|_F=1$ and solution of the system \eqref{eq:gradient_system_structured}. Let $\sigma_j \left( t \right)$ be the simple smallest singular value of the matrix $\mu_j^d \left( A_d + \varepsilon \Delta_d \left( t \right) \right) + \mu_j^{d-1} \left( A_{d-1} + \varepsilon \Delta_{d-1} \left( t \right) \right)+ \ldots + \left( A_0 + \varepsilon \Delta_0 \left( t \right) \right) $ and assume $\sigma_j(t) \neq 0$ for each $j=1,\ldots,m$. Then:
\begin{equation*}
    \frac{d}{dt} G_{\varepsilon} \left( \bDelta \left( t \right) \right) \leq 0.
\end{equation*}
\end{theorem}

\begin{proof}
Since the singular values $\sigma_j(t) \neq 0$ for each $j=1,\ldots,m$, then they are differentiable and we may use Lemma \ref{lemma:derivative_sing_val} in order to express their derivatives. Since $\| \mathbf{\Delta} \|_F=1$, then we observe that
\[
0=\frac{d}{dt} \| \mathbf{\Delta} \|_F^2= 2 {\rm Re} \left\langle \mathbf{\Delta}, \dot{\mathbf{\Delta}} \right\rangle, \; {\rm which \; implies} \; {\rm Re} \left\langle \mathbf{\Delta}, \dot{\mathbf{\Delta}} \right\rangle=0.
\]
Then, using the expression of $\dot{\mathbf{\Delta}}$ in \eqref{eq:gradient_system_structured}, we have:
\begin{equation*}
    \| \dot{\mathbf{\Delta}} \|_F^2 = \left\langle \dot{\mathbf{\Delta}} , \dot{\mathbf{\Delta}} \right\rangle = {\rm Re} \left\langle -\Pi_{\mathcal{S}}\left( \mathbf{M} \right) + \eta \mathbf{\Delta}, \dot{\mathbf{\Delta}} \right\rangle =  -{\rm Re} \left\langle \Pi_{\mathcal{S}}\left( \mathbf{M} \right), \dot{\mathbf{\Delta}} \right\rangle + \eta \underbrace{{\rm Re} \left\langle \mathbf{\Delta} , \dot{\mathbf{\Delta}}\right\rangle}_{=0},
\end{equation*}
from which we derive that
\begin{equation}
\label{eq:formula_for_derivative_G}
    \frac{d}{dt} G_{\varepsilon}\left( \mathbf{\Delta} \right)= \varepsilon {\rm Re} \left\langle \Pi_{\mathcal{S}}\left(\mathbf{M} \right), \dot{\mathbf{\Delta}} \right\rangle = - \varepsilon \| \dot{\mathbf{\Delta}} \|_F^2 \leq 0,
\end{equation}
where the first equality is given by \eqref{eq:derivative_functional_G_projected}.
\end{proof}

The stationary points of the gradient system \eqref{eq:gradient_system_structured} are possible minimizers for the functional $G_{\varepsilon} \left( \bDelta \right)$ and can be characterized in the following way:

\begin{theorem}
\label{thm:equivalence_charact}
Consider $\bDelta(t)$ of unit Frobenius norm, satisfying the gradient system \eqref{eq:gradient_system_structured}. Assume that for all $t$, it holds $G_{\varepsilon} \left( \bDelta \right) >0$. Moreover, suppose that the singular values $\sigma_j$ of $\mu_j^{d} \left( A_d + \varepsilon \Delta_d \right) + \ldots + \left( A_0 + \varepsilon \Delta_0 \right)$ are simple and non-zero, for $j=1,\ldots,m$, and consider $u_j,v_j$ the associated left and right singular vectors, respectively. Then the following are equivalent:
\begin{enumerate}
    \item[(i)] $\frac{d}{dt} G_{\varepsilon}  \left( \bDelta \right) =0$;
    \item[(ii)] $\dot{\bDelta}=0$;
    \item[(iii)] $\bDelta$ is a real multiple of the matrix $\Pi_{\mathcal{S}}\left( \bold{M} \right)$.
\end{enumerate}
\end{theorem}

\begin{proof}
\emph{(iii)}$\implies$\emph{(ii)}: from $\eta={\rm Re}\left\langle \Pi_{\mathcal{S}}\left( \mathbf{\Delta} \right), \mathbf{\Delta} \right\rangle$ and $\| \mathbf{\Delta}\|_F=1$, we get  $\Pi_{\mathcal{S}}\left(\mathbf{M}\right) = \eta \mathbf{\Delta}$. Then using \eqref{eq:gradient_system_structured}, we obtain $\dot{\mathbf{\Delta}}= 0$.

\emph{(ii)}$\implies$\emph{(iii)}: using the expression in \eqref{eq:gradient_system_structured}, we get $0=\dot{\mathbf{\Delta}}= - \Pi_{\mathcal{S}}\left( \mathbf{M}\right) + \eta \mathbf{\Delta}$, from which we get $\Pi_{\mathcal{S}}\left( \mathbf{M} \right) = \eta \mathbf{\Delta}$.

\emph{(ii)}$\implies$\emph{(i)}: using the manipulation in \eqref{eq:formula_for_derivative_G} and the fact that $\dot{\mathbf{\Delta}}=0$, we get that $\frac{d}{dt} G_{\varepsilon}(\mathbf{\Delta})= 0$. 

\emph{(i)}$\implies$\emph{(iii)}: from the expression  \eqref{eq:derivative_functional_G_projected}, we get
\begin{align*}
\frac{1}{\varepsilon} \frac{d}{dt} G_{\varepsilon}(\mathbf{\Delta}) = \mbox{Re}\left\langle \Pi_{\mathcal{S}}(\mathbf{M}), \dot{\mathbf{\Delta}} \right\rangle &= \mbox{Re}\left\langle \Pi_{\mathcal{S}}(\mathbf{M}), -\Pi_{\mathcal{S}}(\mathbf{M}) + \eta \mathbf{\Delta} \right\rangle \\
&= -\|\Pi_{\mathcal{S}}(\mathbf{M}) \|^2_F + \left(\mbox{Re}\left\langle \Pi_{\mathcal{S}}(\mathbf{M}), \mathbf{\Delta} \right\rangle\right)^2 \leq 0,
\end{align*}
where we use that $\eta=\mbox{Re}\left\langle \Pi_{\mathcal{S}}(\mathbf{M}), \mathbf{\Delta} \right\rangle$ and the last inequality holds using the Cauchy-Schwarz inequality and the property $\| \mathbf{\Delta}\|_F=1$. Moreover, the inequality is strict unless we have
\[
\Pi_{\mathcal{S}}(\mathbf{M})= \eta \mathbf{\Delta}, \quad \eta =\pm \| \Pi_{\mathcal{S}}(\mathbf{M}) \|_F.
\]
\end{proof}

The following proves the important property that the structured gradient cannot vanish. In fact, if the gradient vanished we could not characterize stationary points of the gradient system \eqref{eq:gradient_system_structured} as multiple of $\Pi^\cS (\bold{M})$.
\begin{lemma} [Non-vanishing structured gradient]
\label{lem:nonzero-projected-gradient}
Let   $\bold{A},\bDelta\in \cS$ and $\eps>0$, and let 
\begin{equation*}
\sigma_j = \sigma_{\min} \left( \mu_j^d \left( A_d +  \varepsilon 
 \Delta_d \right) + 
    \ldots + \left( A_0 + \varepsilon \Delta_0 \right) \right), \quad j=1,\ldots,m,
\end{equation*}
$\sigma_{\min}$ denoting the smallest singular value.
 
Let $\cS$ be a complex/real-linear subspace of $\C^{\left(d +1 \right)n \times n}$ to which the matrices $\bold{A}$ and $\bDelta$ belong. 
Then, 
\[
\Pi_{\mathcal{S}} \left( \bold{M} \right) \ne 0 \quad\text{ if } \quad \sum\limits_{j=1}^{m} \sigma_j \ne 0.
\]
\end{lemma}

\begin{proof}
    We take the  inner product of $\Pi_{\mathcal{S}} \left( \bold{M} \right)$ with $\bold{A}+\eps \bDelta\in \cS$: 
    \begin{align*}
    &\big\langle \Pi_{\mathcal{S}} \left( \bold{M} \right), \bold{A}+\eps \bDelta \big\rangle =
    \big\langle \bold{M},\bold{A}+\eps \bDelta \big\rangle 
    \\
    &=  \sum\limits_{k=0}^{d} \Big\langle M_k,A_k+\eps \Delta_k \Big\rangle =
    \sum\limits_{k=0}^{d} \Big\langle  \sum_{j=1}^m \bar{\mu_j}^k \sigma_j u_j v_j^H, A_k + \eps \Delta_k \Big\rangle =
    \\
    &= \sum_{j=1}^m\Bigl(  \sum\limits_{k=0}^{d} \mu_j^k u_j^H(A_k+\eps \Delta_k) v_j \Bigr) = 
    \sum\limits_{j=1}^{m} \sigma_j,
    \end{align*}
    where the first equality is derived using the property expressed in formula \eqref{eq:prop_orthogonal_projection} (this can be done since $\mathbf{A} + \varepsilon \mathbf{\Delta}$ belongs to $\mathcal{S}$).
    This yields the result.
\end{proof}

\subsection{The outer iteration}
\label{subsection: outer iteration_struct}

We are interested in computing the smallest value of $\varepsilon$ for which the perturbed matrix polynomial $P\left( \lambda \right) + \Delta P \left( \lambda \right)$ is singular. This means we search for
$\varepsilon^* = \min \left\lbrace \varepsilon > 0 : \; g \left( \varepsilon \right)=0 \right\rbrace$,
where $g \left( \varepsilon \right):=G_{\varepsilon} \left( \bDelta\left( \varepsilon \right) \right)$
and $\bDelta\left( \varepsilon \right)$ is a smooth path of stationary points of \eqref{eq:gradient_system_structured}. The idea consists of tuning the value of $\varepsilon
$, in order to approximate $\varepsilon^*$. We perform a Newton-bisection method, as proposed in \cite{GugLubMeh}. We make the following generic assumption:

\begin{assum}
\label{Assumption_smooth}
The smallest singular values $\sigma_j \left( \varepsilon \right)$ of the matrices 
\[
\mu_j^d \left( A_d + \varepsilon \Delta_d \left( \varepsilon \right) \right) + \mu_j^{d-1} \left( A_{d-1} + \varepsilon \Delta_{d-1} \left( \varepsilon \right) \right) +\ldots + \left( A_0 + \varepsilon \Delta_0 \left( \varepsilon \right) \right)
\]
are simple and different from zero and the functions $\sigma_j \left( \varepsilon \right)$, $\Delta_i \left( \varepsilon \right)$ are smooth w.r.t. $\varepsilon$, for $j=1,\ldots,m$ and $i=0,\ldots,d$.
\end{assum}
This assumption allows us to provide an expression for the derivative of $G\left( \varepsilon \right)$ with respect to $\varepsilon$. Indeed the expression of the derivative of the functional $g(\varepsilon)=G_{\varepsilon}(\mathbf{\Delta}(\varepsilon))$ involves the derivatives of the singular values $\sigma_j(\varepsilon)$ for $j=1,\ldots,m$ w.r.t. $\varepsilon$. The hypotheses on the simplicity and the non-vanishing property on the singular values $\sigma_j(\varepsilon)$ allow us to apply Lemma \ref{lemma:derivative_sing_val}.

\begin{theorem}
\label{thm:Derivative_wrt_eps}
Let $P\left( \lambda \right)= \lambda^d A_d + \lambda^{d-1} A_{d-1} + \ldots + A_0$ a regular matrix polynomial in $\mathbb{C}^{n \times n}$, and $P\left( \lambda \right) + \Delta P \left( \lambda \right) = \lambda^d \left( A_d + \varepsilon \Delta_d \left( \varepsilon \right) \right) + \lambda^{d-1} \left( A_{d-1} + \varepsilon \Delta_{d-1} \left( \varepsilon \right) \right) + \ldots + \left( A_0 + \varepsilon \Delta_0 \left( \varepsilon \right) \right)$, which satisfies the following properties:
\begin{enumerate}
    \item[(i)] $\varepsilon \in \left( 0, \varepsilon^* \right)$, with $g \left( \varepsilon  \right) >0$;
    \item[(ii)] Assumption $\ref{Assumption_smooth}$ holds;
\end{enumerate}
Then:
\begin{equation*}
    \frac{d}{d \varepsilon} g \left( \varepsilon \right) = - \|\Pi_{\mathcal{S}}\left(\bold{M}\left( \varepsilon \right)\right) \|_F.
\end{equation*}
\end{theorem}

\begin{proof}
Since by assumption the singular values $\sigma_j(\varepsilon)$ are simple and different from zero for $j=1,\ldots,m$, we can employ Lemma \ref{lemma:derivative_sing_val} and obtain for $j=1,\ldots,m$:
\begin{align*}
   \frac{1}{2} \frac{d}{d\varepsilon} \sigma_j^2 (\varepsilon) &= \sigma_j(\varepsilon) \frac{d}{d\varepsilon} \sigma_j(\varepsilon) \\
   &= \left( \mbox{Re} \left\langle\sigma_j(\varepsilon) \Bar{\mu_j}^d u_j  v_j^H , {\Delta}_d(\varepsilon)  + \varepsilon \frac{d}{d \varepsilon}{\Delta}_d(\varepsilon)  \right\rangle +\ldots \right.\\
   &\left. + \ldots + \mbox{Re} \left\langle \sigma_j(\varepsilon) u_j v_j^H , {\Delta}_0(\varepsilon) + \varepsilon\frac{d}{d \varepsilon}{\Delta}_0(\varepsilon)  \right\rangle \right),
\end{align*}
where we drop the dependence w.r.t. $\varepsilon$ from $u_j(\varepsilon),v_j(\varepsilon)$ the right and left singular vectors associated with $\sigma_j$, respectively. Similarly to \eqref{eq:derivative_functional} and \eqref{eq:derivative_functional_G_projected}, we can derive:
\begin{equation}
\label{eq:derivative_G_wrt_eps_proof}
    \frac{d}{d\varepsilon} G(\varepsilon)= \mbox{Re} \left\langle \Pi_{\mathcal{S}}(\mathbf{M}(\varepsilon)), \mathbf{\Delta}(\varepsilon) + \varepsilon \frac{d}{d\varepsilon} \mathbf{\Delta}(\varepsilon) \right\rangle
\end{equation}
By definition $\mathbf{\Delta}(\varepsilon)$ is a smooth path of stationary points of \eqref{eq:gradient_system_structured}. Then we use the equivalence \emph{(ii)} $\iff$ \emph{(iii)} in Theorem \ref{thm:equivalence_charact} and obtain that
\begin{equation}
\bDelta \left( \varepsilon \right) = \pm \gamma 
\Pi_{\mathcal{S}}\left(\bold{M}\left( \varepsilon \right)\right), \; \mbox{with} \; \gamma=\frac{1}{\|\Pi_{\mathcal{S}}\left(\bold{M}\left( \varepsilon \right)\right) \|_F}.
\label{eq:condition on Delta in der wrt eps}
\end{equation}
Since $\| \bDelta \left( \varepsilon \right) \|_F=1$ for all $\varepsilon$, we have that
\begin{equation}
\label{eq:derivata_Delta_equal_0}
    0= \mbox{Re} \left\langle \bDelta\left( \varepsilon \right) , \frac{d}{d \varepsilon} \bDelta \left( \varepsilon \right) \right\rangle = \gamma \mbox{Re} \left\langle \Pi_{\mathcal{S}}\left(\bold{M}\left( \varepsilon \right)\right), \frac{
    d}{d \varepsilon} \bDelta \left( \varepsilon \right) \right\rangle,
\end{equation}
Substituting the relation \eqref{eq:derivata_Delta_equal_0} in the formula given by \eqref{eq:derivative_G_wrt_eps_proof}, we have
\[
\frac{d}{d\varepsilon}g(\varepsilon)=\frac{d}{d\varepsilon}G(\varepsilon)=\pm \gamma \| \Pi_{\mathcal{S}}(\mathbf{M}) \|_F.
\]
Since by assumption $g(\varepsilon)>0$ for each $\varepsilon< \varepsilon^*$ and $g(\varepsilon^*)=0$ by definition of $\varepsilon^*$, we get that the derivative $\frac{d}{d\varepsilon}g(\varepsilon) <0$, which implies
\[
\frac{d}{d\varepsilon} g(\varepsilon)= - \| \Pi_{\mathcal{S}}(\mathbf{M}) \|_F.
\]
\end{proof}

Using Assumption \ref{Assumption_smooth}, we get that the function $\varepsilon \rightarrow g(\varepsilon)$ is smooth for $\varepsilon < \varepsilon^*$ and its derivative w.r.t. $\varepsilon$ is given in Theorem \ref{thm:Derivative_wrt_eps}. Therefore, we can apply Newton's method:
\begin{equation}
\label{eq:Newton iteration}
    \varepsilon_{k+1} = \varepsilon_k - \left( \| 
    \Pi_{\mathcal{S}}\left(\bold{M}\left( \varepsilon_k \right)\right)\|_F \right)^{-1} g \left( \varepsilon_k \right), \; k =1,\ldots.
\end{equation}
Since local convergence to $\varepsilon^*$ for Newton's method is guaranteed from the left, we need to correct the approach using a bisection method for the approximation by the right.

\begin{remark}
If instead of applying the two-level iteration we just integrated the gradient system 
\eqref{eq:gradient_system_structured} without imposing the norm conservation, that is the system of ODEs
\begin{equation*}
\dot{\bDelta}= - \Pi_{\mathcal{S}} \left( \bold{M} \right) 
\end{equation*}
we would still have a gradient system for the unconstrained problem 
\[
\min\limits_{\bDelta \in \cS} G_{\varepsilon} \left( \bDelta \right) = \frac{1}{2} 
\sum_{j=1}^m \sigma_{j}^2 \left(  \varepsilon, \bDelta \right) = 0
\]
but we would not obtain in general the smallest (in norm) solution $\bDelta$.

Since we are looking for the solution $\bDelta$ of minimum norm,
this explains the use of the two-level formulation of the constrained 
optimization problem. For instance, consider the matrix polynomial $P_{\delta}\left( \lambda \right)$, with $\delta=0.9$ in Example \ref{ex:polyn_near_to_sing}. If we employ the gradient system for the unconstrained problem we obtain an approximation of the distance to singularity of $3.1907$, while the result obtained with our two-level formulation is $0.5461$.

Also note that if $\widehat \bDelta$ is such that $G_{\varepsilon} ( \widehat \bDelta )=0$ then
\[
\bA + \eps \widehat\bDelta + (\bA + \eps \widehat\bDelta) s, \qquad s \in \R
\]
also defines a solution to the unconstrained minimization problem.

This implies the property 
\begin{equation*}
g(\eps) = 0 \qquad \mbox{for} \ \eps \ge \eps^*.
\end{equation*}
\end{remark}

\subsection{Fixed coefficients}
\label{subsec:Fixed_coefficients}

In several situations, some coefficient matrices of $P\left( \lambda \right)$ may derive from specific characteristics of the problem and then it may be not meaningful to perturb these coefficients. Consider a set of indices $I \subset \left\lbrace 0, 1, \ldots , d \right\rbrace$, with cardinality $\left| I \right| \le d$. Then we are interested in computing the distance
\begin{equation*}
d_{\rm sing}^I \left(P  \right):= \min \left\lbrace \| \Delta P \| : P \left( \lambda \right) + \Delta P \left( \lambda \right) \in \mathcal{P}_d  \; \mbox{singular}, \; \Delta A_i =0 \; \forall i \in I \right\rbrace.
\end{equation*}

The method can be applied in the same way minimizing the functional $G_{\varepsilon} \left( \bDelta \right)$, with the constraint 
\begin{equation*}
    \Delta_i(t) \equiv 0 , \; \mbox{for each} \; i \in I.
\end{equation*}

In this case, the algorithm can be modified using a specialization of Lemma \ref{lem: minimization_struct_pert_general} and computing the stationary points of the ODE
\begin{equation*}
    \dot{\Delta_i}= - M_i + \eta \Delta_i, \quad \mbox{for} \; i \not\in I,
\end{equation*}
where $\eta=\sum_{i \not\in I} \mbox{Re}\left\langle \Delta_i, M_i  \right\rangle$.

\subsection{Linear structure}
\label{subsec:Linear_structure}

In this Subsection, we consider the case in which the coefficients of the matrix polynomial $P\left( \lambda \right)$ have an additional structure. In particular, we focus on polynomials whose coefficients are such that
\begin{equation*}
    A_i \in \mathcal{S}_i \; \mbox{for each} \; i=0,\ldots,d.
\end{equation*}
and therefore we consider perturbations $\Delta \bold{A}$ in the subspace
\begin{equation*}
\mathcal{S}:= \left\lbrace \Delta \bold{A}: \; \Delta A_i \in \mathcal{S}_i, \; i=0,\ldots,d \right\rbrace.    
\end{equation*}
For simplicity, we consider linear subspaces $\mathcal{S}_i \subseteq \mathbb{C}^{n \times n}$: for instance, they can be the set of sparse matrices or real matrices. Note that the matrix subspaces $\mathcal{S}_i$ can be both different from each other and all the same: for instance $\mathcal{S}_i= \mathbb{R}^{n \times n}$, for $i=0,\ldots,d$ for the case of polynomials with real coefficients. Then in general for each subspace $\mathcal{S}_i$, we construct a different orthogonal projection $\Pi_{\mathcal{S}_i}$, which preserves the structure for the perturbation matrix $\Delta_i$. In this context, we obtain the constrained gradient system
\begin{equation*}
    \dot{\bDelta}= - \bold{\Pi} + \eta \bDelta,
\end{equation*}
where 
\begin{equation*}
    \bold{\Pi}=\left[ \begin{array}{c}
         \Pi_{\mathcal{S}_d}\left( M_d \right)  \\
          \Pi_{\mathcal{S}_{d-1}}\left( M_{d-1} \right)  \\
          \vdots \\
           \Pi_{\mathcal{S}_0}\left( M_0 \right)
    \end{array}\right]
\end{equation*}
and $\eta=\mbox{Re} \left\langle \bDelta, \bold{\Pi} \right\rangle$. In this setting the distance to singularity in Definition \ref{definition:distance_to_sing_unstr} can be seen as a structured distance to singularity, with subspaces $\mathcal{S}_i=\mathbb{C}^{n \times n}$ for $i=0,\ldots,d$ and where the projections $\Pi_{\mathcal{S}_i}$ are the identity on $\mathbb{C}^{n\times n}$. In the case of sparse matrices, for instance, each coefficient matrix $A_i$ can have its own sparsity pattern. For each $\mathcal{S}_i$, the projection $\Pi_{\mathcal{S}_i}$ is the map that preserves the entries in the sparsity pattern and sets to zero the remaining entries.

\subsection{Structure on the whole polynomial}
\label{subsec:Structure_on_whole_pol}

The third example of structure that we present consists of the presence of additional relations among the coefficients of the matrix polynomial. In this class of polynomials, we can find palindromic, anti-palindromic, symmetric, hermitian, even or odd polynomials. Here we propose the specialization of the method for palindromic polynomials:
\begin{equation*}
    P\left(\lambda \right) \; \mbox{such that} \; P \left( \lambda \right)= \lambda^d P^H \left( 1/\lambda \right). 
\end{equation*}

Therefore in order to preserve the structure of the problem, we restrict the perturbations to the subspace
\begin{equation*}
    \mathcal{S}:= \left\lbrace \Delta \bold{A} : \; \Delta A_{d-i}= \Delta A_i^H, \; \mbox{for } i=0,\ldots,d  \right\rbrace,
\end{equation*}
and consider the projection $\Pi_{\mathcal{S}}: \mathbb{C}^{\left( d+1\right)n \times n }  \mapsto \mathcal{S}$ given by
\begin{align*}
    \left[ \begin{array}{c}
     A_d \\
     \vdots \\
     A_0 \\
    \end{array} \right] & \mapsto \left[ \begin{array}{c} \frac{A_d+A_0^H}{2}\\ 
    \vdots \\
    \frac{A_d^H + A_0}{2} 
    \end{array} \right].
\end{align*}
The specialization of Lemma \ref{lem: minimization_struct_pert_general} leads to the resolution of the gradient system
\begin{equation*}
    \dot{\bDelta}= -\bold{\widehat{M}} + \eta \bDelta,
\end{equation*}
where we denote
\begin{equation*}
    \bold{\widehat{M}} = \left[ \begin{array}{c}
         \frac{M_d +M_0^H}{2} \\
         \vdots \\
         \frac{M_d^H+M_0}{2}\\
    \end{array}\right]
\end{equation*}
and $\eta= \mbox{Re} \left\langle \bDelta,\bold{\widehat{M}} \right\rangle$.

\section{Computational approach and numerical examples}
\label{sec:computationa-approach}

The iterative algorithm is a two-level method.

\subsection*{Inner iteration}

We need to integrate numerically the differential equations {\eqref{eq:gradient_system_structured}.

The objective here is not to follow a particular trajectory accurately, but to arrive quickly at a stationary point.
The simplest method is the normalized Euler method, where the result after an Euler step (i.e., a steepest descent step) 
is normalized to unit norm.
This can be combined with an Armijo-type line-search strategy to determine the step size adaptively (see e.g \cite{Luenberger}).

This algorithm requires in each step the computation of $m$ smallest singular eigenvalues and associated singular 
vectors, which might be computed at reduced computational cost for certain structures, like for large sparse matrices $A_i$, 
by using an implicitly restarted Arnoldi method (as implemented in the MATLAB function {\it eigs}).

\subsection*{Outer iteration}

The outer iteration is performed using a Newton-bisection technique on the interval $\left[ \varepsilon_{\rm low}, \varepsilon_{\rm up} \right]$ (see \cite{GugLubMeh} for more details). One possible initialization for the bounds is $0$ for the lower bound and the norm $\| \bold{A} \|_F$ for the upper one. In our numerical experiments, we choose the starting value $\varepsilon_0$ for the distance $\varepsilon$ equal to zero and the perturbation matrix $\bDelta \left( \varepsilon_0 \right)$ is chosen as the free gradient, scaled in order to have unit norm.

For the outer iteration, we prefer to employ Newton's method on the function $g \left( \varepsilon \right) - {\rm tol}_1$. Formally this change guarantees the quadratic convergence of the method. Then we substitute the iteration \eqref{eq:Newton iteration} with
\begin{equation*}
\label{eq:Newton it_for Alg1}
    \varepsilon_{k+1}= \varepsilon_{k} - \frac{ g\left( \varepsilon_k \right) - {\rm tol}_1}{\| g' \left( \varepsilon_k \right) \|_F}.
\end{equation*}

In our numerical experiments, we take the tolerance ${\rm tol}_1$ to be equal to $d \cdot 10^{-6}$ and the maximum number of iterations equal to $20$. The choice of the set of sample points $\mu_j, \; j=1,\ldots,m$ is a delicate feature of the problem. In our numerical experiments, we choose to use exactly $m=dn+1$ points. We stop the method when the width of the interval $\left[ \varepsilon_{\rm{low}}, \varepsilon_{\rm{up}} \right]$ is below a certain threshold ${\rm{tol}}_2$.

\subsection{Numerical examples}
\label{sec:Numerical examples_no_ker}

In order to show the behaviour of our method, we show a few illustrative examples. For simplicity, we test the method on polynomials of low degrees.

\begin{example}
\label{ex:polyn_near_to_sing}

Consider the singular polynomial $P_{1} \left( \lambda \right)$ in \eqref{pol_P1_introduction}. Here we impose as an additional structure on the perturbation matrices the sparsity pattern induced by the initial matrices. A slight perturbation in one of the coefficients can produce a regular matrix polynomial, for instance:
\begin{equation*}
    P_{1}^{\delta} \left( \lambda \right)= \lambda^2 \left[ \begin{array}{c c}
    1  & 0 \\
    0  & 0
 \end{array} \right] + \lambda \left[ \begin{array}{c c}
    0  & 1 \\
   1 - \delta  & 0
 \end{array} \right] + \left[ \begin{array}{c c}
    0  & 0 \\
    0  & 1
 \end{array} \right].
\end{equation*}

Setting different values of $\delta$, we can test the behaviour of the method. In Table \ref{table:ex1_generic}, we collect the numerical results. We set the tolerances as ${\rm tol}_1=5 \cdot 10^{-7}$ and ${\rm tol}_2= 10^{-7}$. 
\begin{table}[!h]
\tbl{Approximation $\varepsilon^*$ of the structured distance to singularity for $P_{1}^{\delta}$.}
{\begin{tabular}{| c| c| c|}
\hline
$\delta$ & $\varepsilon^*$ & iterations \\
\hline \hline
$0.9$ & $5.4614 \cdot 10^{-1}$ & $14$ \\
$0.5$ & $ 2.8033 \cdot 10^{-1}$ & $14$ \\
$10^{-1}$ & $5.0802 \cdot 10^{-2}$ & $11$ \\
$10^{-2}$ & $4.5653 \cdot 10^{-3}$ & $9$ \\
\hline
\end{tabular}}
\label{table:ex1_generic}
\end{table}

\end{example}

\begin{example}
Consider the matrix polynomial
\begin{equation*}
P\left( \lambda \right)=\lambda^2 \left[ \begin{array}{c c c}
0 & 0 & 0  \\
0 & 0 & 1 \\
0 & 1 & 0 \\
\end{array} \right] + \lambda \left[ \begin{array}{c c c}
1 & 0 & 0 \\
0 & 1 & 0 \\
0 & 0 & 1\\
\end{array} \right] + \left[ \begin{array}{c c c}
     0  &  0.4 &  0.89 \\
    0.15 &  -0.02 &  0\\
    0.92 &  0.11 &   0.06\\
\end{array} \right].
\end{equation*}
To illustrate the method in Subsection \ref{subsec:Fixed_coefficients}, we do not perturb the coefficient matrix of $\lambda^2$. Then the computed distance to singularity with fixed $A_2$ is $\varepsilon^*=1.2415$ and the perturbations are the following (truncated to $4$ digits):
\begin{align*}
    \varepsilon^* \Delta_1&= \left[ \begin{array}{c c c}
   -0.9992 & -0.0983 & 0.0636 \\
   0.0732 & -0.2296 &  0.2111 \\
   0.0506 & -0.1282 & 0.1166 \\
    \end{array} \right], \quad 
    \varepsilon^* \Delta_0 = \left[ \begin{array}{c c c}
   -0.0006 &  0.2020 & -0.1940\\
   -0.4901 & -0.0459 & -0.0129 \\
  -0.2674 & 0.0120 &  -0.0397 \\
    \end{array} \right].
\end{align*}
Note that the computation of the distance to singularity without any additional constraints on the coefficients and therefore allowing perturbations for all the matrices gives an approximation of the distance equal to $d_{\mbox{sing}}\approx 1.1054$.
\end{example}

\begin{example}
Consider the matrix polynomial
\begin{equation*}
P\left( \lambda \right)=
    \lambda^2 \left[ \begin{array}{c c c}
      1  & 0 & 0\\
     0 & 1  & 0 \\
     0 & 0 & 1 \\
    \end{array} \right] + 
 \lambda \left[ \begin{array}{c c c}
    -1.79 &  0.10 & -0.60\\
    0.84 & -0.54 & 0.49\\
    -0.89 & 0.30 & 0.74\\
    \end{array} \right] +
    \left[ \begin{array}{c c c}
    0 & 0 & 0 \\
    0 & 0 & 1 \\
    0 & 1 & 0\\
    \end{array} \right].
\end{equation*}
If we do not impose additional constraints and allow complex perturbations, the method produces an approximation of the distance to singularity of $d_{\mbox{sing}} \approx \varepsilon^* =  1.2775967141$ and the complex perturbations of the coefficients, with elements truncated to $4$ digits, are the following:
\begin{align*}
    \varepsilon^*\Delta_2 &= \left[ \begin{array}{c  c c}
  -0.2620 - 0.0062\textbf{i}  &  -0.2781 - 0.1703\textbf{i} &  0.0437 - 0.2691\textbf{i} \\
  -0.2853 + 0.1642\textbf{i} & -0.4380 + 0.0089\textbf{i} &  -0.1363 - 0.3456\textbf{i}\\
   0.0658 + 0.2717\textbf{i} & -0.1408 + 0.3493\textbf{i} & -0.3005 - 0.0139\textbf{i}\\
    \end{array}  \right], \\[2mm]
    \varepsilon^*\Delta_1 &=\left[ \begin{array}{c c c}
    0.1179 - 0.0006\textbf{i}  &-0.0695 + 0.0030\textbf{i} &  0.0380 - 0.0081\textbf{i}\\
   0.0809 - 0.0053\textbf{i}  &-0.0432 + 0.0031\textbf{i} &  0.0217 + 0.0009\textbf{i}\\
  -0.1128 - 0.0107\textbf{i}  & 0.0746 + 0.0014\textbf{i} & -0.0452 + 0.0160\textbf{i}\\
    \end{array}  \right],\\[2mm]
    \varepsilon^*\Delta_0 &=\left[ \begin{array}{c c c}
  0.0183 + 0.0030\textbf{i}  & 0.3563 - 0.0080\textbf{i}  &-0.3250 + 0.0027\textbf{i}\\
   0.0828 + 0.0068\textbf{i} &  0.2267 - 0.0076\textbf{i} & -0.2173 - 0.0032\textbf{i}\\
   0.0946 + 0.0106\textbf{i} & -0.3670 - 0.0079\textbf{i} &  0.3194 - 0.0004\textbf{i}\\
\end{array} \right].
\end{align*}
Here we set the tolerances as ${\rm tol}_1=7 \cdot 10^{-7}$ and ${\rm tol}_2= 10^{-7}$. 

The presence of real coefficients suggests the employment of the extension of the method to structured perturbations, as described in Subsection \ref{subsec:Linear_structure}. Since in this case, we have $\mathcal{S}_i= \mathbb{R}^{3 \times 3}$ for $i=0,1,2$, then the orthogonal projection consists in taking the real part of the matrices. The numerical approximation given by the method is $d_{\mbox{sing}}^{\mathbb{R}}\approx \varepsilon^*=1.2927804886$ and the final real perturbation are given by: 
\begin{align*}
    \varepsilon^*\Delta_2 &= \left[ \begin{array}{c c c}
      -0.1537  & -0.2687 &  -0.1738\\
   -0.2668  & -0.5360  & -0.4148\\
   -0.1718 &  -0.4121  & -0.3823\\
    \end{array}  \right], \\
    \varepsilon^*\Delta_1 &= \left[ \begin{array}{c c c}
   0.1095  & -0.1392 & 0.0120\\
    0.0331  & -0.0891   & 0.0349\\
   -0.1434   & 0.1030 & 0.0358\\
    \end{array}  \right],\\
    \varepsilon^*\Delta_0 &=\left[ \begin{array}{c c c}
    0.0485 &   0.3778  & -0.2605\\
    0.1296  &  0.0999  & -0.1466\\
    0.1305  & -0.5158  &  0.2254\\
\end{array} \right].
\end{align*}
\end{example}

\begin{example}
Consider the palindromic matrix polynomial
\begin{equation*}
 P\left(\lambda \right)= \lambda^2 \left[ \begin{array}{c c c}
     0   &  0  & -1\\
     1   & -1 &   -1\\
     0   & -2  &   0\\
 \end{array} \right] + \lambda \left[ \begin{array}{c c c}
      -1 &  0 & -0.5\\
       0 & 0  & -0.5\\
   -0.5 & -0.5 & 2 \\
 \end{array} \right] + \left[ \begin{array}{c c c}
     0   &  1  &  0\\
     0   & -1 &  -2\\
    -1   & -1  &   0\\
 \end{array} \right].
\end{equation*}
We apply the approach in Subsection \ref{subsec:Structure_on_whole_pol} and obtain that the approximate distance to singularity is equal to $d_{\mbox{sing}}^\mathcal{S} \approx 1.0523$ and the structured singular matrix polynomial is given by the following perturbations:
\begin{align*}
    \varepsilon^* \Delta_2 &= \left[  \begin{array}{c c c}
    0.5242 &  0.1556  &  0.1637\\
   -0.2821 &  -0.0333 &  0.0205\\
    0.1499  & -0.0336 & 0.0325\\
\end{array}   \right] = \varepsilon^* \Delta_0^T, \\
\varepsilon^* \Delta_1 &= \left[ \begin{array}{c c c}
     0.3091 & -0.1264 &  0.2020\\
   -0.1264  & -0.1749 & -0.0279\\
    0.2020  & -0.0279 &  0.0532\\
\end{array} \right] = \varepsilon^* \Delta_1^T.
\end{align*}

\end{example}

\begin{example}
The five points relative pose problem arises in the context of computer vision. It consists of the reconstruction of the position of two cameras starting from the images of five unknown scene points taken from the two viewpoints. As described in \cite{KukBujPaj}, the problem can be rephrased into a cubic polynomial eigenvalue problem in the form
\begin{equation*}
    P\left( \lambda \right)=A_3 \lambda^3 + A_2\lambda^2 + A_1 \lambda +A_0, \in \mathbb{R}^{10 \times 10},
\end{equation*}
where $A_3$ has rank $1$, $A_2$ has rank $3$, $A_1$ has rank $6$ and $A_0$ has full rank. We use the polynomial contained in the MATLAB toolbox \texttt{nlevp} \cite{BetHigTiss}. We consider real perturbations, therefore we impose that $\mathcal{S}_i=\mathbb{R}^{10 \times 10}$. Our method computes an approximation of the distance to the nearest real singular matrix polynomials of $d_{\rm sing} \approx 0.1116$. The method gives the result in $11$ iterations. The final matrices $A_i + \varepsilon ^* \Delta_i$, for $i=0,1,2,3$ do not present common kernel. In the end, the most perturbed matrix is the coefficient $A_0$.
\end{example}

\section{Computing the nearest singular polynomial with common kernel}
\label{sec:distance_for_common_kernel}

As we have mentioned, when the singularity is determined by the presence of a common kernel among the coefficients of the polynomial, it is more appropriate to compute the nearest matrix polynomial whose coefficients have a common kernel. An approach for the detection of the nearest pencil with common null vectors has been proposed in \cite{GugLubMeh} and then specialized for dissipative Hamiltonian pencils in \cite{GugMehr}. Their method consists of the construction of an optimization problem, solved by a constrained gradient system. This minimization problem involves the computation of left and right eigenvectors associated with the smallest eigenvalue of the pencil, and consequently, of their derivatives. In particular, the formula for the derivatives of the eigenvectors involves the computation of the group inverse, as proved in \cite{MeyStew}. In our method, instead, we avoid this step with the introduction of a different minimization functional. In this way, the method is different and significantly simpler than the one proposed in \cite{GugLubMeh} for matrix pencils. Our approach works in the same way both for right kernels and left kernels. For simplicity, we describe the theory only for right kernels, but the same process can be repeated in an analogous way for left ones.

\begin{definition}
Consider a matrix polynomial $P\left( \lambda \right)$ as in \eqref{def:Polynomial_P}. The unstructured distance to the nearest polynomial whose coefficients have a common kernel is defined as:
\begin{equation*}
 d_{\mbox{ker}} \left( P  \right) = \min \left\lbrace \| \Delta P  \| : \; \exists x \in \mathbb{C}^{n}, \, x \neq 0 : \left( A_i  + \Delta A_i \right) x=0  \right\rbrace.
\end{equation*}
\end{definition}

It can be proved that the distance $d_{\rm ker}\left( P \right)$ is equal to the smallest singular value of $\bold{A}$ and therefore the problem allows a direct solution. The structured problem, instead, is more challenging and requires careful analysis. For this motivation, we prefer to show our approach for the case of polynomials with additional structures. In these contexts, we are interested in numerically approximating the structured distance to the nearest common kernel.

\begin{definition}
\label{def:distance_to_common_kernel_Structured}
Given a subset $\mathcal{S} \subseteq \mathbb{C}^{\left( d+1\right) n \times n}$. The structured distance to the nearest polynomial whose coefficients have a common kernel is given by
\begin{equation*}
    d_{\rm ker}^{\mathcal{S}} \left( P \right):=\min \left\lbrace \| \Delta P \|: \; \exists x \in \mathbb{C}^{n}, \; x \neq 0 : \; \left( A_i + \Delta A_i \right)x =0, \; \Delta \bold{A} \in \mathcal{S} \right\rbrace.
\end{equation*}
\end{definition}

 In order to formulate an optimization problem as in Section \ref{sec: section ODE based approach}, we recall the following equivalence:

\begin{lemma}
\label{lem:Equivalence_common_kernel}
Consider $P\left( \lambda \right) \in \mathcal{P}_d$. Then the following are equivalent:
\begin{enumerate}
    \item[(i)] there exists $x \in \mathbb{C}^n$, $x \neq  0$ such that $A_i x =0$, for $i=0,\ldots,d$;
    \item[(ii)] the smallest singular value of $\bold{A}$ is equal to zero.
\end{enumerate}
\end{lemma}

\begin{proof}
Let us prove $(i) \Rightarrow (ii)$. From $(i)$, we get that 
\begin{equation*}
\bold{A} x=0,
\end{equation*}
which means that the columns are not linearly independent, and therefore the rank is not full. Then we conclude that $(ii)$ holds true.

Vice versa, from $(ii)$ we have that the matrix
\begin{equation*}
    \left[ A_d^H \; \cdots \; A_0^H \right] \left[ \begin{array}{c}
        A_d  \\
        \vdots \\
        A_0
    \end{array} \right]= A_d^H A_d + \ldots + A_0^H A_0
\end{equation*}
has zero eigenvalue. Consider $x \neq 0$ the right eigenvector associated with $0$, then we have that
\begin{equation*}
    x^H \left( A_d^H A_d + \ldots + A_0^H A_0 \right)x= \| A_d x\|_F^2+ \ldots + \| A_0 x \|_F^2=0,
\end{equation*}
from which we get that $A_ix=0$, for $i=0,\ldots,d$.
\end{proof}
Using Lemma \ref{lem:Equivalence_common_kernel}, we formulate the following optimization problem:
\begin{align}
\label{eq:opt_pbm_d_ker}
    \Delta \bold{A}^* =& \mbox{arg} \min_{\Delta \bold{A} \in \mathcal{S}}\| \Delta \bold{A} \|_F \\
    &{\rm subj. \; to} \; \sigma_{\min} \left(\bold{A}+ \Delta \bold{A} \right)=0, \notag
\end{align}
where $\sigma_{\min}$ denotes the smallest singular value of the matrix. Using this formulation, we have that $d_{\rm ker}^{\mathcal{S}} \left( P \right)=\| \Delta \bold{A}^* \|_F$. 

Fixing the perturbation in the form $\Delta \bold{A}= \varepsilon \bDelta$, with $\| \bDelta \|_F=1$ and $\bDelta \in \mathcal{S}$, we construct an optimization problem equivalent to \eqref{eq:opt_pbm_d_ker} minimizing the functional
\begin{equation}
\label{def:functional_F}
    F_{\varepsilon} \left( \bDelta \right) = \frac{1}{2} \sigma_{\min}^2 \left( \bold{A}+ \varepsilon \bDelta \right),
\end{equation}
where $\sigma_{\min}$ is the smallest singular value of the matrix $\bold{A} + \varepsilon \bDelta$. In order to solve the problem, we apply the steepest descent method and impose the condition \eqref{eq:Condition_conservation_norm} for the conservation of the norm of the perturbation matrices.
The derivative of $F_{\varepsilon} \left( \bDelta \right)$ can be computed using Lemma \ref{lemma:derivative_sing_val} and written through the Frobenius inner product as:
\begin{equation}
\label{eq:derivative_F}
   \frac{1}{\varepsilon \sigma} \frac{d}{dt} F_{\varepsilon} \left( \bDelta \right)= 
   {\rm Re} \left\langle u v^H, \dot{\bDelta} \right\rangle= 
   {\rm{Re}}\left\langle \Pi_{\mathcal{S}}\left( uv^H \right), \dot{\bDelta} \right\rangle,
\end{equation}
where $\sigma = \sigma \left( t \right)$ is the smallest singular value of $\bold{A}+ \varepsilon \bDelta \left( t \right)$ 
and $u, v$ are the associated left and right singular vectors, respectively. Thus we interpret $\Pi_{\mathcal{S}}\left( uv^H \right)$ as the free gradient of 
$F_{\varepsilon}$.

Again we prove now a sequence of results that allow us to optimize
the functional \eqref{def:functional_F} and characterize the stationary 
points of the associated gradient system \eqref{eq:gradient_system_com_ker} as (local) extremizers.

The first result (Lemma \ref{lem:minimiz_common kernel}) provides the gradient of the constrained functional, which is followed by its non-vanishing property (Lemma \ref{lem:nonzero-gradient-S}). 

Theorem \ref{thm: monotonicity of F} states that along solution trajectories of \eqref{eq:gradient_system_com_ker} the functional \eqref{def:functional_F} decreases monotonically in $t$.

Theorem \ref{thm:equiv_charac_common_kern} characterizes stationary points
of the system \eqref{eq:gradient_system_com_ker} as real multiples of the 
free gradient $\Pi_{\mathcal{S}} \left( u v^H \right)$, which provides us
(local) minimizers.

The steepest descent direction for minimizing the functional \eqref{def:functional_F} is given by the following lemma:
\begin{lemma}
\label{lem:minimiz_common kernel}
Consider $\bold{Z},\bDelta \in \mathcal{S}$. A solution of the minimization problem
\begin{align*}
    \bold{Z}^*=& {\rm arg} \min_{\bold{Z} \in \mathcal{S} } {\rm Re} \left\langle \Pi_{\mathcal{S}} \left( u v ^H \right), \bold{Z} \right\rangle \\
    &{\rm subj.\; to} \; {\rm Re} \left\langle \bDelta, \bold{Z} \right\rangle=0, \; {\rm and} \; \| \bDelta \|_F=1 \notag
\end{align*}
is given by
\begin{equation*}
    \nu \bold{Z}^*= - \Pi_{\mathcal{S}} \left( u v^H \right) + \zeta \bDelta,
\end{equation*}
where $\zeta= {\rm Re} \left\langle \Pi_{\mathcal{S}}\left( u v^H \right), \bDelta  \right\rangle$ 
and $\nu$ is the norm of the right-hand side.
\end{lemma}

\begin{proof}
The proof is analogous to the one proposed for Lemma \ref{lem: minimization_struct_pert_general}, using the identification among $\mathbb{C}^{n\left(d+1\right)\times n}$ and $\mathbb{R}^{2n\left(d+1\right)\times 2n}$ and performing the orthogonal projection of the matrix $-\Pi_{\mathcal{S}}(uv^H)$ to the orthogonal complement of the span of $\mathbf{\Delta}$.
\end{proof}

Similarly to Lemma \ref{lem:nonzero-projected-gradient} we have the following result.
\begin{lemma} [Non-vanishing structured gradient]
\label{lem:nonzero-gradient-S}
Let   $\bold{A},\bDelta\in \cS$, a complex/real-linear subspace of $\C^{\left( d+1\right) n \times n}$, $\eps>0$ and $\sigma$  smallest singular value of $\bold{A}+\eps \bDelta$.
Then, 
\[
\Pi_{\mathcal{S}} \left( uv^H \right) \ne 0 \quad\text{ if } \quad \sigma \ne 0.
\]
\end{lemma}

\begin{proof}
    Again we  consider the inner product 
    \begin{align*}
    &\big\langle \Pi_{\mathcal{S}} \left( uv^H \right), \bold{A}+\eps \bDelta \big\rangle = 
    \big\langle uv^H,\bold{A}+\eps \bDelta \big\rangle = u^H(\bold{A}+\eps \bDelta) v  = \sigma.
    \end{align*}
    Exploiting $\sigma \ne 0$ concludes the proof.
\end{proof}

Using Lemma \ref{lem:minimiz_common kernel}, we have that the following
\begin{equation}
\label{eq:gradient_system_com_ker}
    \dot{\bDelta}= - \Pi_{\mathcal{S}} \left( uv^H \right) +  \zeta \bDelta,
\end{equation}
is a constrained gradient system. Moreover, along each solution of \eqref{eq:gradient_system_com_ker}, the functional $F_{\varepsilon} \left( \bDelta \left( t \right) \right)$ decreases monotonically.

\begin{theorem}
\label{thm: monotonicity of F}
Consider $\bDelta \left( t \right) \in \mathcal{S}$ of unit Frobenius norm and solution of \eqref{eq:gradient_system_com_ker}. Let $\sigma \left( t \right)$ the simple smallest singular value of $\bold{A}+ \varepsilon \bDelta \left( t \right)$ and assume $\sigma(t) \neq 0$. Then:
\begin{equation*}
    \frac{d}{dt} F_{\varepsilon} \left( \bDelta \left( t \right) \right) \leq 0.
\end{equation*}
\end{theorem}

\begin{proof}
The proof follows using \eqref{eq:derivative_F} and \eqref{eq:gradient_system_com_ker}, similarly to Theorem \ref{thm:G_decreases}, employing the functional $F_{\varepsilon}(\mathbf{\Delta})$ and the matrix $\Pi_{\mathcal{S}}(uv^H)$ in place of $G_{\varepsilon}(\mathbf{\Delta})$ and $\Pi_{\mathcal{S}}(\mathbf{M})$, respectively.
\end{proof}

The stationary points of the matrix ODE \eqref{eq:gradient_system_com_ker} can be characterized as follows.

\begin{theorem}
\label{thm:equiv_charac_common_kern}
Consider the functional \eqref{def:functional_F} and suppose $F_{\varepsilon}\left( \bDelta \right) >0$. Consider $\sigma$ simple non-zero singular value of $\bold{A}+ \varepsilon \bDelta$, with $u,v$ associated left and right singular vectors, respectively. Then the following statements are equivalent:
\begin{enumerate}
    \item[(i)] $\frac{d}{dt} F_{\varepsilon} \left( \bDelta \right)=0$;
    \item[(ii)] $\dot{\bDelta}=0$;
    \item[(iii)] $\bDelta$ is a real multiple of the matrix $\Pi_{\mathcal{S}} \left( u v^H \right)$.
\end{enumerate}
\end{theorem}

\begin{proof}
    The equivalence can be proved in the same as in Theorem \ref{thm:equivalence_charact}, exploiting the expression for the derivative in \eqref{eq:derivative_F} and the gradient system in \eqref{eq:gradient_system_com_ker}.
\end{proof}

Using the results in Theorem \ref{thm:equiv_charac_common_kern}, we get that the perturbation matrix we are looking for is a rank one matrix in the form
$
    \bDelta= - {\Pi_{\mathcal{S}} \left( uv^H\right)}/{\|\Pi_{\mathcal{S}} \left( u v^H \right) \|_F}.
$
From Theorem \ref{thm:equiv_charac_common_kern}, we observe that the stationary points of \eqref{eq:gradient_system_com_ker} are real multiples of a projection of rank one matrices. This piece of information may be exploited in the algorithm in order to improve its efficiency. Several attempts in this direction have been recently proposed by Guglielmi et al. in \cite{GugLubSic}, with a focus on structures induced by sparsity patterns. The outer iteration, instead, can be developed in the same way as Subsection \ref{subsection: outer iteration_struct}, using a hybrid Newton-bisection algorithm on the functional
\begin{equation*}
    f\left( \varepsilon \right):= F_{\varepsilon} \left( \bDelta\left( \varepsilon \right)  \right),
\end{equation*}
where $\bDelta\left( \varepsilon \right)$ is the stationary point of unit Frobenius norm given by the integration of \eqref{eq:gradient_system_com_ker}.
A relevant setting of application of the method for structured perturbations may be the one of DAE systems with dissipative Hamiltonian structures. In this context, indeed, the pencil associated with the system is in the form
\begin{equation*}
    L \left( \lambda \right):=\lambda A_2 - \left( A_1-A_0 \right),
\end{equation*}
with $A_i \in \mathbb{R}^{n \times n}$ for $i=0,1,2$, $A_1=-A_1^T$, and $A_2,A_0$ symmetric and positive definite. Since the coefficient matrices are constrained to have this dissipative Hamiltonian structure, then it is convenient to analyze the distance to the nearest structured triplet with common coefficients, as proposed in \cite{GugMehr}. In this case, the singularity occurs when the matrices $A_2,A_1,A_0$ have a common kernel, then we can apply our method relying only on the coefficient matrices. The subspace in which we search for perturbations is
\begin{equation*}
 \mathcal{S}:= \left\lbrace \Delta \bold{A}: \; \Delta A_i = \Delta A_i^T,\; \mbox{for} \; i=0,2, \; \mbox{and} \; \Delta A_1=-\Delta A_1^T  \right\rbrace.
\end{equation*}

Then, starting from the gradient system \eqref{eq:gradient_system_com_ker}, we specialize the method with the computation of the stationary points of
\begin{align*}
    \dot{\Delta}_i &= -\mbox{Sym} \left( uv^H \right) + \zeta \Delta_i, \quad i=0,2, \\ 
     \dot{\Delta}_1 &= -\mbox{SkewSym} \left( uv^H \right) + \zeta \Delta_1 \notag,
\end{align*} 
where Sym and SkewSym denote the symmetric and skew-symmetric parts of a matrix, respectively.

\subsection{Numerical examples}
\label{sec:numericalexamples}

In this Subsection we provide some illustrative examples.

\begin{example}
\label{ex: example_comm_kern_GuglMehr}
Consider the quadratic matrix polynomial $P \left( \lambda \right)=\lambda^2 A_2 + \lambda A_1 + A_0 \in \mathbb{R}^{5 \times 5}$, taken from \cite{GugMehr}, where the coefficients are
\begin{align*}
    A_2  &=\left[ \begin{array}{c c c c c}
    0.15 & 0.02 & -0.04 & 0.02 & -0.04\\
     0.02 & 0.22 & 0  & -0.01 & -0.03\\
     -0.04 & 0 & 0.11 &  -0.07 & -0.04\\ 
    0.02 & -0.01 & -0.07 & 0.01 & 0.10\\
     -0.04 & -0.03 & -0.04 & 0.10 & 0.39 \\
    \end{array} \right] \\
    A_1 &= \left[ \begin{array}{c c c c c}
    0 & -0.27 & -0.03 & -0.01 & 0.21\\
     0.27 &  0 & -0.15 & 0.03 & 0.11\\
     0.03 & 0.15 &  0  & 0.07 & -0.07\\ 
     0.01 & -0.03 & -0.07 & 0 & 0.05\\
     -0.21 & -0.11 & 0.07 & -0.05 & 0\\
    \end{array} \right] \\
    A_0 &= \left[ \begin{array}{c c c c c}
     0.49 & -0.13 & 0.05 & -0.15 & 0.11\\
     -0.13 &  0.23  & -0.05 & -0.10 & -0.19\\
     0.05 & -0.05 & 0.48 & -0.06 & 0.02\\ 
     -0.15 & -0.10 & -0.06 & 0.55 & 0.16\\
     0.11 & -0.19 & 0.02 & 0.16 & 0.48\\
    \end{array}
    \right].
\end{align*}
Note that the coefficients $A_2,A_0$ are symmetric matrices and $A_1$ is skew-symmetric. The method proposed in Section \ref{sec:distance_for_common_kernel} produces an approximation of the structured distance $d_{\rm ker}^{\mathcal{S}} \approx \varepsilon^*=0.3541658817$ with perturbation matrices 
\begin{align*}
    \varepsilon^* \Delta_2 &= \left[ 
    \begin{array}{c c c c c}
    -0.0105  &  0.0305  & 0.0033  &  0.0057  &  0.0421\\
    0.0305  & -0.0806 &  -0.0070  & -0.0340  & -0.1280\\
    0.0033 &  -0.0070 &  -0.0001 &  -0.0077 &  -0.0153\\
    0.0057 &  -0.0340 &  -0.0077 & 0.0362 &  -0.0095\\
    0.0421  & -0.1280  & -0.0153  & -0.0095  & -0.1643\\
    \end{array}
    \right], \\
    \varepsilon^* \Delta_1 &= \left[ 
    \begin{array}{c c c c c}
     0   & 0.0423 &  -0.0004 &  -0.0153 &   0.0399\\
   -0.0423   &  0 &   0.0271  & -0.0013 &  -0.0081\\
    0.0004 &  -0.0271  &   0  &  0.0098 &  -0.0255\\
    0.0153  &  0.0013  & -0.0098  &  0  &  0.0041\\
   -0.0399  &  0.0081  &  0.0255  & -0.0041  &  0\\
    \end{array}
    \right], \\
    \varepsilon^* \Delta_0 &= \left[ 
    \begin{array}{c c c c c}
   -0.0198 &   0.0425  &  0.0113 &  -0.0268  &  0.0434\\
    0.0425 &  -0.0387 &  -0.0231  &  0.0561  & -0.0554\\
    0.0113  & -0.0231  & -0.0064  &  0.0153  & -0.0240\\
   -0.0268  &  0.0561  &  0.0153  & -0.0362  &  0.0577\\
    0.0434  & -0.0554  & -0.0240  &  0.0577 & -0.0680\\
    \end{array}
    \right].
\end{align*}
The method proposed in \cite{GugMehr} produces a numerical approximation for $d_{\rm ker}^{\mathcal{S}}$ equal to $0.3568$ (rounded to four digits) and different perturbation matrices.
\end{example}

\begin{example}
The singularity associated with a common kernel among the coefficients may also occur in quadratic polynomials arising in linear gyroscopic systems in the form:
\begin{equation*}
    M \ddot{x} \left( t \right) + G \dot{x} \left( t \right) + K x \left( t \right)=0, \quad M,G,K \in \mathbb{R}^{n \times n},
\end{equation*}
where the mass matrix $M$ is symmetric and positive definite, the gyroscopic matrix $G$ is skew-symmetric and $K$, which is related to potential forces, is symmetric. The structure of the coefficient matrices derives directly from the properties of gyroscopic systems. Then the coefficients belong to the set
\begin{equation*}
    \mathcal{S}:=\left\lbrace \bold{A} \in \mathbb{R}^{3n \times n} : \; A_2, A_0 \; {\rm symmetric}, \; A_1 \; {\rm skew \; symmetric} \right\rbrace .
\end{equation*}
We compute the approximation of the distance $d_{\rm ker}^{\mathcal{S}}$ for right common kernel, using the polynomial $\lambda^2 M + \lambda G + K \in \mathbb{R}^{3 \times 3}$ in \cite{GuglManet}, with coefficients $M=I$
\begin{equation*}
    G=\left[ \begin{array}{c c c}
    0 & -2 & 4\\
    2 & 0 & -2\\
    -4 & 2 & 0  \end{array} 
    \right], \quad K= \left[  \begin{array}{c c c}
         13 & 2 &1\\
         2 & 7 &2\\
        1 & 2 &4 \end{array} \right].
\end{equation*}
The approximation of $d_{\rm ker}$ is given by $\varepsilon^*=14.4976$. This result shows that the polynomial is particularly far from the singularity, when the perturbation matrices are constrained to the subspace $\mathcal{S}$.
The obtained extremal perturbations are:
\begin{align*}
    \varepsilon^*\Delta_2 &= \left[ \begin{array}{c c c}
    -0.0018 &  0.0051  &  0.0005\\
    0.0051  & -0.0260 &  -0.0168\\
    0.0005  & -0.0168 &  -0.9997\\
    \end{array} \right], \\
    \varepsilon^*\Delta_1 &= \left[ \begin{array}{c c c}
     0  & 8.2412  & -4.1079\\
   -8.2412 &  0  &  1.9966\\
    4.1079  & -1.9966 & 0\\
    \end{array} \right], \\
    \varepsilon^*\Delta_0 &= \left[ \begin{array}{c c c}
   0.0360 & 0.4187 & -1.0347\\
    0.4187 &   1.9680 & -2.1538\\
   -1.0347  & -2.1538  & -3.9974\\
    \end{array} \right].
\end{align*}

\end{example}

\section{A posteriori bounds to singularity} 
\label{sec:a posteriori}

From a numerical point of view, it is not possible to solve $G_\eps(\bDelta) = 0$ exactly, but only to approximate 
it with a tiny error. This means that the computed polynomial is \emph{numerically singular} but not exactly.
An important issue is certainly how to relate this inaccuracy with singularity of the perturbed polynomial. 
Following the methodology we propose, we say that a computed polynomial $P+\Delta P$ such that the associated functional $G_\eps(\bDelta) < {\rm{tol}}$ (with a tiny tolerance ${\rm{tol}}$) is  \emph{numerically singular}.
Does this mean that there exists a nearby singular polynomial, that is that the distance to singularity of what we have computed is indeed tiny? The question we formulate here is then the following:

\begin{itemize}
    \item Given a numerically singular polynomial (possibly preserving the structure of the original polynomial),
    what is an upper bound to its distance to singularity?
\end{itemize}

The program of this section is the subsequent. Given a polynomial computed by our methodology, which is numerically singular, that is such that the considered functional $G_\eps$ is made tiny enough, i.e. below a given threshold, we propose to make use of some available
upper bounds for its (unstructured) distance to singularity, that is to the closest singular polynomial. 
These bounds can be computed in a reliable way and provide us an a-posteriori measure that is able to give an answer to the question we have set. 
If the bound turns out to be tiny then we consider the obtained result reliable. A few examples will illustrate this analysis.

A posteriori bounds on the accuracy of our approximation can be found comparing the outcome with upper bounds for the distance to singularity coming from purely algebraic approaches. One of these has been developed in \cite{ByHeMeh} for matrix pencils. Consider the perturbations $\bold{\Delta}$ arising in our approach. We associate to our perturbed matrix polynomial the linearization obtained with the companion formulation and apply the QZ algorithm of the pencil. It has been proved by Van Dooren in \cite{VanDoDew} that this approach is backward stable, using the definition of backward error
\begin{equation*}
    \eta=\min\left\lbrace \varepsilon: \left( P(\widetilde{\lambda}) + \Delta P(\widetilde{\lambda}) \right)\widetilde{x}=0, \; \|\bold{\Delta A} \|_F \leq \varepsilon \|\bold{A} \|_F \right\rbrace.
\end{equation*}
Therefore, if the upper bound obtained using the QZ is in the order of the considered tolerance, then our approach produces a polynomial close to a singular one. 

For a more precise result, it is possible to employ Theorem $2.1$ in \cite{Dmytryshyn} (or Theorem $5.21$ \cite{Dopico}). Indeed, given a matrix polynomial $P(\lambda)$ and its first companion form $\mathcal{C}_{P(\lambda)}$, it can be proved that, given a sufficiently small in norm pencil $\mathcal{E}$ as a perturbation, there exists a matrix polynomial $\Delta P(\lambda)$ for which $\mathcal{C}_{P(\lambda)} + \mathcal{E}$ is strictly equivalent to the linearization of $\mathcal{C}_{P(\lambda)+\Delta P(\lambda)}$ and moreover we have the following bound on the distance:
\[
\| \mathcal{C}_{P(\lambda)+\Delta P(\lambda)} - \mathcal{C}_{P(\lambda)} \| \leq 4d(1 + \| P(\lambda)\|_F) \|\mathcal{E}\|.
\]

In the very recent paper \cite{DasBora}, the authors propose a characterization of the nearest singular matrix polynomial to a regular one $P(\lambda)=\sum_{i=0}^d A_i \lambda^i$, using convolution matrices 
\[
C_{r}\left( P \right) = \underbrace{\begin{bmatrix}
    A_0 &  & & \\
    A_1 & A_0 & & \\
    \vdots & A_1  & \ddots & \\
    A_d & \vdots  & \ddots & A_0 \\
     & A_d & & A_1 \\
     & & \ddots & \vdots \\
     &  & & A_d\\
\end{bmatrix}}_{r+1 \; \mbox{block columns}}
\]
for $r=0,1,\ldots$. In particular, the characterization of the unstructured distance to singularity given in Theorem $6.3$ (pag. $319$ in \cite{DasBora}) may be useful to provide a few upper bounds for this distance. Indeed the authors derive the following upper bound for the distance to singularity in the Frobenius norm for a matrix polynomial $P$ of size $n \times n$ and degree $k$:
\begin{equation}
\label{eq:upper_bound_DasBora}
    d_{\mbox{sing}}(P) \leq \min \left\lbrace \min_{0 \leq j \leq k(n-1)}\frac{\sigma_{\min} \left( C_j(P) \right)}{\sigma_{\min}\left( V_j\right)} ,  \min_{0 \leq j \leq k(n-1)}\frac{\sigma_{\min} \left( C_j(P^T) \right)}{\sigma_{\min} ( \widetilde V_j )} \right\rbrace,
\end{equation}
where $V_j$ and $\widetilde V_j$ are the matrices
\begin{align*}
    V_j &= \left. \begin{bmatrix}
        v_0 & v_1 & \cdots & v_{j-1}&  v_j & & & \\
         & v_0 & v_1 & \cdots & v_{j-1} & v_j & & \\ 
          &  & \ddots & \ddots & \ddots & \ddots & \ddots & \\ 
         &  &  & v_0 & v_1 & \cdots & v_{j-1} & v_j\\ 
    \end{bmatrix} \right\rbrace  (k+1)n  \; \text{rows}, \\
    \widetilde V_j &= \left. \begin{bmatrix}
        \tilde v_0 & \tilde v_1 & \cdots & \tilde v_{j-1}& \tilde v_j & & & \\
         & \tilde v_0 & \tilde v_1 & \cdots & \tilde v_{j-1} & \tilde v_j & & \\ 
          &  & \ddots & \ddots & \ddots & \ddots & \ddots & \\ 
         &  &  & \tilde v_0 & \tilde v_1 & \cdots & \tilde v_{j-1} & \tilde v_j\\ 
    \end{bmatrix}  \right\rbrace  (k+1)n  \; \text{rows},
\end{align*}
defined using the right singular vector $\left[ v_0^T \; \cdots \; v_j^T \right]$ of $C_j(P)$ associated with the smallest singular value $\sigma_{\min}\left( C_j(P)\right)$ and the right singular vector $\left[ \tilde v_0^T \; \cdots \; \tilde v_j^T  \right]$ of $C_j(P^T)$ associated with the smallest singular value $\sigma_{\min}\left( C_j(P^T)\right)$. It is important to observe that the upper bound appears to be tight a many cases (see Section $9$ in \cite{DasBora}).

We observe that the upper bound in \eqref{eq:upper_bound_DasBora} provides a limitation for the unstructured distance to singularity since the nearby singular polynomial does not necessarily preserve the structure of the
computed polynomial. Nevertheless, it can be a useful tool in order to verify a-posteriori that the perturbed matrix polynomial $P(\lambda)+ \Delta P(\lambda)$ computed by our approach is indeed close to being singular.

\subsubsection*{Application to Example \rm \ref{ex:polyn_near_to_sing}}

Consider again Example \ref{ex:polyn_near_to_sing} with $\delta=10^{-2}$. Running our method with the additional constraint of the sparsity pattern induced by the coefficients, we obtain the following perturbed polynomial (here reported with $4$ digits precision):
\[
P_{1}^{\delta}(\lambda) + \varepsilon^*\Delta(\lambda) = \lambda^2 \begin{small}\begin{bmatrix}
    9.9749 \times 10^{-1} & 0 \\
     0 & 0\\
\end{bmatrix}\end{small} + \lambda \begin{small} \begin{bmatrix}
   0 & 1.0025 \\
    9.9252 \times 10^{-1} & 0\\
\end{bmatrix}\end{small} + \begin{small}\begin{bmatrix}
   0 & 0 \\
    0 & 9.9749 \times 10^{-1}\\
\end{bmatrix}\end{small}.
\]
In order to check that the perturbed polynomial $P_1^{\delta}(\lambda) + \Delta(\lambda)$ is indeed close to singular, we compute the upper bound \eqref{eq:upper_bound_DasBora} for the polynomial $P_1^{\delta}(\lambda) + \Delta(\lambda)$ and in this case this is equal to $1.4142 \times 10^{-6}$. This is sufficient to assure that the perturbed polynomial is at maximum $1.4142 \times 10^{-6}$ away from a singular one. Note that usually the tolerance required in the inner iteration of our method is in the order of $10^{-6}$.
\medskip

If the numerical singularity of the polynomial, which we detect by applying our algorithm, is associated with the presence of a common kernel among the coefficient matrices, we have that the unstructured distance to singularity has a direct solution, taking the smallest singular value of the matrix $\bold{A}$. Then in these contexts we may have an upper bound on the unstructured distance to singularity of $P(\lambda)$ given by
\begin{equation}
\label{eq:upper_bound_kernel}
    d_{\rm sing}(P) \leq \min \left\lbrace \sigma_{\min} ( \left[A_0^T, \ldots, A_d^T\right]^T ) ,\sigma_{\min}\left( \left[A_0, \ldots, A_d \right] \right) \right\rbrace ,
\end{equation}
as proposed by \cite{ByHeMeh} for pencils and \cite{DasBora} for matrix polynomials. Therefore, for this class of problems, it is reasonable to test if the matrix polynomial obtained with our approach is singular using both \eqref{eq:upper_bound_DasBora} and \eqref{eq:upper_bound_kernel}. 

\subsubsection*{Application to Example \rm \ref{ex: example_comm_kern_GuglMehr}} 
Consider again Example \ref{ex: example_comm_kern_GuglMehr}.
We may check if the final structured polynomial with common kernel ${P}(\lambda) + \varepsilon \Delta (\lambda)$ is close to a singular one. The upper bound in \eqref{eq:upper_bound_DasBora} applied to the perturbed matrix polynomial is equal to $3.6356 \times 10^{-4}$, while the upper bound given by \eqref{eq:upper_bound_kernel} is $4.7016 \times 10^{-4}$. This shows that ${P}(\lambda) + \varepsilon \Delta (\lambda)$ is at maximum $3.6356 \times 10^{-4}$ away from a singular polynomial.

\section{Conclusions and further work} \label{sec:concl}
We have proposed an approach for the numerical approximation of the distance to singularity for regular matrix polynomials. We have considered two different cases, depending on the possible presence of a common left/right kernel among the coefficient matrices, and developed two related numerical methods, both based on the computation of the gradient flow of a suitable functional. Remarkably, our approach allows us to compute structured distances, where the coefficient matrices belong to a linear subspace. The structures we consider may involve not only the coefficient matrices individually, but also the whole set, as in the palindromic case. Further studies are needed in order to speed up the algorithm, with particular focus on the computational speed of the ODE integrator, which constitutes the most expensive part of the whole algorithm. 

\section*{Acknowledgements}

We wish to thank two anonymous Referees for their helpful suggestions and constructive remarks. In particular we thank an anonymous Referee for suggesting to add Section \ref{sec:a posteriori} to the original version of the article.

Miryam Gnazzo and Nicola Guglielmi thank the INdAM GNCS (Gruppo Nazionale di Calcolo Scientifico).

\section*{Disclosure statement}
No potential conflict of interest was reported by the authors.

\section*{Funding}

Nicola Guglielmi acknowledges that his research was supported by funds from the Italian 
MUR (Ministero dell'Universit\`a e della Ricerca) within the PRIN 2017 Project 
`Discontinuous dynamical systems: theory, numerics and applications'' 
and the   Pro3 joint project entitled
``Calcolo scientifico per le scienze naturali, sociali e applicazioni: sviluppo metodologico e tecnologico''.

\bibliographystyle{tfnlm}
\bibliography{biblio_revision}

\end{document}